%% file: main.tex
\documentclass[oneside,english]{amsart}
\usepackage{amsthm}
\usepackage{wasysym}
\usepackage{graphicx}
\usepackage{setspace}
\usepackage{amssymb}

\makeatletter
\numberwithin{equation}{section} 
\numberwithin{figure}{section} 
\theoremstyle{plain}
\theoremstyle{plain}
\newtheorem{thm}{Theorem}
  \theoremstyle{plain}
  \newtheorem{lem}[thm]{Lemma}
  \theoremstyle{plain}
  \newtheorem{cor}[thm]{Corollary}

\numberwithin{thm}{section}
\usepackage{wrapfig}

\thanks{HB and WSL were supported in part by grants from the National Science Foundation.}

\address{HB: Department of Mathematics, Indiana University, Bloomington, IN 47405, USA}

\email{bercovic@indiana.edu}

\address{WSL: School of Mathematics, Georgia Institute of Technology, Atlanta, GA 30332-1060, USA}

\email{li@math.gatech.edu}

\address{DT: Simion Stoilow Institute of Mathematics of the Romanian Academy, PO Box 1-764, Bucharest 014700, Romania}

\email{Dan.Timotin@imar.ro}

\makeatother

\usepackage{babel}

\begin{document}

\title{A family of reductions for Schubert intersection problems}

\author{H. Bercovici, W. S. Li and D. Timotin}

\maketitle
\markboth{}{}
\begin{abstract}
We produce a family of reductions for Schubert intersection problems
whose applicability is checked by calculating a linear combination
of the dimensions involved. These reductions do not alter the Littlewood-Richardson
coefficient, and they lead to an explicit solution of the intersection
problem when this coefficient is 1.
\end{abstract}

\section{Introduction}

Given integers $n>r\ge1$, we denote by $G(r,\mathbb{C}^{n})$ the
Grassmannian manifold consisting of all $r$-dimensional subspaces
in $\mathbb{C}^{n}$. For every flag\[
\mathcal{E}=\{\{0\}=\mathbb{E}_{0}\subset\mathbb{E}_{1}\subset\mathbb{E}_{2}\subset\cdots\subset\mathbb{E}_{n}=\mathbb{C}^{n}\},\]
where $\mathbb{E}_{j}$ is a subspace of dimension $j$, $G(r,\mathbb{C}^{n})$
can be written as a union of Schubert varieties described as follows.
For each set $I=\{i_{1}<i_{2}<\cdots<i_{r}\}\subset\{1,2,\dots,n\}$
one defines the Schubert variety\[
\mathfrak{S}(\mathcal{E},I)=\{\mathbb{M}\in G(r,\mathbb{C}^{n}):\dim(\mathbb{M}\cap\mathbb{E}_{i_{x}})\ge x,x=1,2,\dots,r\}.\]
Schubert calculus allows one to find the number of points in the intersection
of several Schubert varieties $\mathfrak{S}(\mathcal{E}_{\ell},I_{\ell}),\ell=1,2,\dots,p$,
when the flags $(\mathcal{E}_{\ell})_{\ell=1}^{p}$ are in \emph{generic
}position. We will be mostly concerned with the case $p=3$ where
the classical Littlewood-Richardson rule applies (cf. \cite{fulton}).
Thus, given sets $I,J,K\subset\{1,2,\dots,n\}$ of cardinality $r$
such that\[
\sum_{\ell=1}^{r}(i_{\ell}+j_{\ell}+k_{\ell}-3\ell)=2r(n-r),\]
the Littlewood-Richardson rule (which will be reviewed below) provides
a non-negative integer $c_{IJK}$ with the property that the set\[
S=\mathfrak{S}(\mathcal{E},I)\cap\mathfrak{S}(\mathcal{F},J)\cap\mathfrak{S}(\mathcal{G},K)\]
contains $c_{IJK}$ elements for generic flags $\mathcal{E},\mathcal{F},\mathcal{G}$.
For nongeneric flags, this intersection is still certain to be nonempty
if $c_{IJK}>0$.

Thompson and Therianos \cite{th-th} pointed out that under certain
circumstances one can reduce the problem of finding elements in the
set $S$ to a problem where $n$ is replaced by a smaller number.
In order to explain their reductions, it will be convenient to set
$I=\{i_{1}<i_{2}<\cdots<i_{r}\}$, define $i_{0}=0$, and similarly
$j_{0}=k_{0}=0$. Assume that the indices $x,y,z\in\{0,1,2,\dots,r\}$
are such that $x+y+z=r$ and $i_{x}+j_{y}+k_{z}=n-p<n.$ In this case,
the spaces $\mathbb{E}_{i_{x}},\mathbb{F}_{j_{y}},\mathbb{G}_{k_{z}}$
are generically independent, and for any space $\mathbb{M}\in S$
we have\[
r=\dim(\mathbb{M})\ge\dim(\mathbb{M}\cap\mathbb{E}_{i_{x}})+\dim(\mathbb{M}\cap\mathbb{F}_{j_{y}})+\dim(\mathbb{M}\cap\mathbb{G}_{k_{z}})\ge x+y+z=r.\]
Therefore $\mathbb{M}$ is contained in $\mathbb{E}_{i_{x}}+\mathbb{F}_{j_{y}}+\mathbb{G}_{k_{z}}$.
Replace now $\mathbb{C}^{n}$ by the space $\mathbb{X}=\mathbb{E}_{i_{x}}+\mathbb{F}_{j_{y}}+\mathbb{G}_{k_{z}}$
and the spaces $\mathbb{E}_{i},\mathbb{F}_{j},\mathbb{G}_{k}$ by
their intersections with $\mathbb{X}$. Observe that generically\[
\dim(\mathbb{E}_{i}\cap\mathbb{X})=\begin{cases}
i & \text{if }i\le i_{x}\\
i_{x} & \text{if }i_{x}<i\le i_{x}+p\\
i-p & \text{if }i_{x}+p<i\le n,\end{cases}\]
and these spaces will form (after the repeating spaces of dimension
$i_{x}$ are deleted) a flag $\mathcal{E}'$ in $\mathbb{X}$. Flags
$\mathcal{F}'$ and $\mathcal{G}'$ are defined similarly. Finding
the spaces in $S$ amounts to finding the spaces in \[
S'=\mathfrak{S}(\mathcal{E}',I')\cap\mathfrak{S}(\mathcal{F}',J')\cap\mathfrak{S}(\mathcal{G}',K')\subset G(r,\mathbb{X}),\]
where\[
i'_{\ell}=\begin{cases}
i_{\ell} & \text{if }1\le\ell\le x\\
i_{\ell}-p & \text{if }x<\ell\le r,\end{cases}\]
with similar definitions for $J'$ and $K'$. (The sequence $(i'_{\ell})_{\ell=1}^{r}$
is still strictly increasing because the condition $i_{x}+j_{y}+k_{z}=n-p$
actually implies that $i_{x+1}>i_{x}+p$.) The question arises naturally
whether $c_{I'J'K'}\ne0$ if $c_{IJK}\ne0$, so that the reduced problem
is still guaranteed to have a solution. That this is indeed the case
was shown by Collins and Dykema \cite{CoDy-reduction} who proved
that in fact $c_{I'J'K'}=c_{IJK}$.

The purpose of this paper is to identify a much larger family of reductions
associated with various inequalities satisfied by $I,J,K$. This family
is sufficient for the complete solution of the intersection problem
when $c_{IJK}=1$. The simplest of these new reductions is as follows.
Assume that $x,y,z\in\{1,2,\dots,r\}$ satisfy $x+y+z=2r$ and\[
i_{x}+j_{y}+k_{z}=2n-p<2n.\]
In this case the space\[
\mathbb{X}=(\mathbb{E}_{i_{x}}\cap\mathbb{F}_{j_{y}})+(\mathbb{E}_{i_{x}}\cap\mathbb{G}_{k_{z}})+(\mathbb{F}_{j_{y}}\cap\mathbb{G}_{k_{z}})\]
has generically codimension $2p$ and it contains all the spaces in
$S$. The reduced problem in $G(r,\mathbb{X})$ corresponds with the
sets $I',J',K'$ defined by\[
i'_{\ell}=\begin{cases}
i_{\ell}-p & \text{if }1\le\ell\le i_{x}\\
i_{\ell}-2p & \text{if }i_{x}<\ell\le r,\end{cases}\]
with analogous definitions for $j'_{\ell},k'_{\ell}$. As in the result
of \cite{CoDy-reduction} just mentioned, we have $c_{I'J'K'}=c_{IJK}$.
The general reduction we propose can be described as follows. We are
given $r$-tuples $a=(a_{\ell})_{\ell=1}^{r},b=(b_{\ell})_{\ell=1}^{r},c=(c_{\ell})_{\ell=1}^{r}$
of nonegative integers such that\[
\sum_{\ell=1}^{r}(\ell a_{\ell}+\ell b_{\ell}+\ell c_{\ell})=\omega r\]
for some positive integer $\omega$; $a,b,c$ are subject to other
conditions which will be discussed later. Assume that the sets $I,J,K\subset\{1,2,\dots,n\}$
have cardinality $r$, $c_{IJK}>0$, and consider the sum\begin{equation}
\sum_{\ell=1}^{r}(a_{\ell}i_{\ell}+b_{\ell}j_{\ell}+c_{\ell}k_{\ell})=\omega n-p,\label{eq:reduction-witness}\end{equation}
where $p$ is some integer. The reduction corresponding to $a,b,c$
can be applied when $p<0$. Namely, if $p<0$, we necessarily have
$\omega p\le n$. Moreover, there exist
\begin{enumerate}
\item a space $\mathbb{X}\subset\mathbb{C}^{n}$ with $\dim\mathbb{X}=n-\omega p$,
\item flags $\mathcal{E}',\mathcal{F}',\mathcal{G}'$ in $\mathbb{X}$,
\item sets $I',J',K'\subset\{1,2,\dots,n-\omega p\}$ of cardinality $r$
such that $c_{I'J'K'}=c_{IJK}$ and\[
\mathfrak{S}(\mathcal{E}',I')\cap\mathfrak{S}(\mathcal{F}',J')\cap\mathfrak{S}(\mathcal{G}',K')\subset\mathfrak{S}(\mathcal{E},I)\cap\mathfrak{S}(\mathcal{F},J)\cap\mathfrak{S}(\mathcal{G},K).\]

\end{enumerate}
In addition, the space $\mathbb{X}$ can be constructed (when the
flags $\mathcal{E},\mathcal{F},\mathcal{G}$ are in `general position')
explicitly from $\mathcal{E},\mathcal{F},\mathcal{G}$ by applying
a finite number of sums and intersections. The sequences $a,b,c$
which appear here are themselves related to the Littlewood-Richardson
rule.

The two reductions discussed above are such that the only nonzero
components of $a,b,c$ are $a_{x}=b_{y}=c_{k}=1$, and $\omega=1$
or $\omega=2$.

Our proofs deepen some of the results in \cite{bcdlt}. Even though
we review the relevant results of \cite{bcdlt}, familiarity with
that paper would be helpful in reading this one.

The remainder of the paper is organized as follows. In Section 2 we
describe the formulation of the Littlewood-Richardson rule in terms
of measures. This is essentially the puzzle formulation of \cite{KTW},
and was also used in \cite{bcdlt}. We also introduce the linear combinations
of dimensions which serve as witnesses for the possibility of reductions.
In Section 3 we discuss a special class of measures, the tree measures.
It was implicit in the results of \cite{bcdlt} that rigid extremal
measures have an underlying tree structure, and this is made explicit
here. Section 4 reviews the construction of a puzzle from a measure,
and uses the results of Section 3 to deduce the identity $c_{I'J'K'}=c_{IJK}$.
In Section 5 we prove the essential technical result needed to show
in Section 6 that the analogues of the reductions of \cite{th-th}
can indeed be performed. It seems practically impossible to describe
all rigid tree measures in a uniform manner. We provide in Section
7 a description of a fairly large class of such measures.

\section{The Littlewood-Richardson Rule}

We will give the description of the Littlewood-Richardson rule in
terms of measures. This is equivalent with the puzzle description
of \cite{KTW}. Choose unit vectors $u,v,w$ in the plane such that
$u+v+w=0$.

\begin{figure}[htbp]
\begin{center} 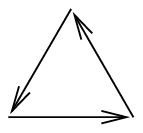  

\end{center}
\end{figure} 

\noindent The points $iu+jv$ with integer $i,j$ will be called \emph{lattice
points}, and a segment joining two nearest lattice points will be
called a \emph{small edge}. We consider positive measures $m$ which
are supported by the union of the small edges, whose restriction to
each small edge is a multiple of arclength measure, and which satisfy
the balance condition (called zero tension in \cite{KTW})\begin{equation}
m(AB)-m(AB')=m(AC)-m(AC')=m(AD)-m(AD')\label{eq:balance}\end{equation}
whenever $A$ is a lattice point and the neighboring lattice points
$B,C',D,B',C,D'$ are in cyclic order around $A$.

\begin{figure}[htbp]
\begin{center} 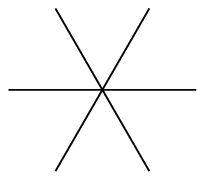  

\end{center}
\end{figure}If $e$ is a small edge, the value $m(e)$ is equal to the density
of $m$ relative to arclength measure on that edge.

Fix now an integer $r\ge1$, and denote by $\triangle_{r}$ the (closed)
triangle with vertices $0,ru,$ and $ru+rv=-rw$. We will use the
notation $A_{j}=ju,B_{j}=ru+jv$, and $C_{j}=(r-j)w$ for the lattice
points on the boundary of $\triangle_{r}$. We also set \[
X_{j}=A_{j}+w,Y_{j}=B_{j}+u,Z_{j}=C_{j}+v\]
for $j=0,1,2,\dots,r+1$. The following picture represents $\triangle_{5}$
and the points just defined; the labels are placed on the left.
\begin{center}

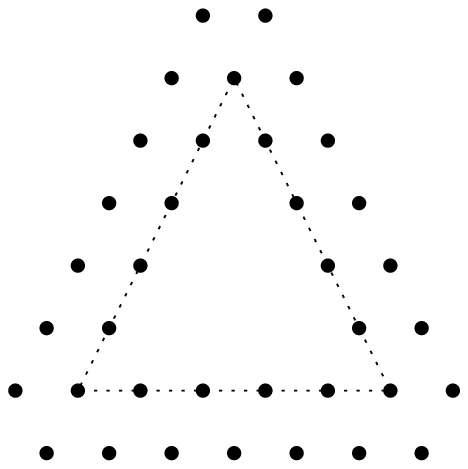

\end{center}
to at least three edges in the support of $m$. We will only consider
measures with at least one branch point. This excludes measures whose
support consists of one or more parallel lines. We  denote by $\mathcal{M}_{r}$
the collection of all measures $m$ satisfying the balance condition
above, whose branch points are contained in $\triangle_{r}$, and
such that\[
m(A_{j}X_{j+1})=m(B_{j}Y_{j+1})=m(C_{j}Z_{j+1})=0,\quad j=0,1,\dots,r.\]
The numbers $\alpha_{j}=m(A_{j}X_{j})$, $\beta_{j}=m(B_{j}Y_{j})$
and $\gamma_{j}=m(C_{j}Z_{j})$ will be called the \emph{exit densities
}of $m$.\emph{ }The \emph{weight} $\omega(m)$ of a measure $m\in\mathcal{M}_{r}$
is defined as\[
\omega(m)=\sum_{j=0}^{r}\alpha_{j}=\sum_{j=0}^{r}\beta_{j}=\sum_{j=0}^{r}\gamma_{j};\]
the equality of the three sums follows from the balance condition. 

Assume that $m\in\mathcal{M}_{r}$ assigns integer densities to all
small edges. We can then define an integer \[
n=r+\omega(m),\]
and sets $I,J,K\subset\{1,2,\dots,n\}$ of cardinality $r$ by setting
$I=\{i_{1},i_{2},\dots,i_{r}\}$, where \begin{equation}
i_{\ell}=\ell+\sum_{j=0}^{\ell-1}\alpha_{j},\quad\ell=1,2,\dots,r,\label{eq:15}\end{equation}
with similar formulas for $J$ and $K$. These are precisely the triples
of sets $(I,J,K)$ which satisfy the Littlewood-Richardson rule. The
\emph{Littlewood-Richardson coefficient} $c_{IJK}$ equals the number
of measures $m\in\mathcal{M}_{r}$ with integer densities which satisfy
(\ref{eq:15}). (See \cite{KTW}, or \cite[Appendix]{buch} for a
direct proof of this fact.) We will also write $c_{m}=c_{IJK}$ when
$I,J,K$ are obtained from $m$. When $c_{m}=1$, we will say that
$m$ is \emph{rigid.} In other words, $m$ is rigid if there is no
other measure with the same exit densities. Note that knowledge of
$n$ and of the sets $I,J,K$ determines entirely the numbers $\alpha_{j},\beta_{j},\gamma_{j}$.
The Littlewood-Richardson rule requires these numbers to be the actual
exit densities of some measure.

One of the advantages of this formulation of the Littlewood-Richardson
rule is that it displays an underlying convexity structure. Thus,
the set $\mathcal{M}_{r}$ is a convex polyhedral cone, and therefore
each measure $0\ne m\in\mathcal{M}_{r}$ can be written as a sum of
\emph{extremal} measures. Recall that $m\ne0$ is extremal if every
measure $m'\le m$ is a multiple of $m$. This decomposition into
extremal summands is unique (except for the order of the terms) if
$m$ is a rigid measure (see \cite[Corollary 3.6]{bcdlt}). In the
proof of Theorem \ref{thm:rigid-skeletons-are-trees}, we will describe
briefly the result of \cite{bcdlt} showing how the extremal summands
of a rigid measure are obtained.

The results of \cite{KTW} imply that, given a measure $m\in\mathcal{M}_{r}$
with exit densities $\alpha_{j},\beta_{j},\gamma_{j}$, there exist
Hermitian $r\times r$ matrices $X,Y,Z$ such that $X+Y+Z=2\omega(m)1_{r}$,
and the eigenvalues of $X,Y,Z$ are, respectively, the numbers\[
\sum_{j=0}^{\ell-1}\alpha_{j},\quad\sum_{j=0}^{\ell-1}\beta_{j},\quad\sum_{j=0}^{\ell-1}\gamma_{j},\quad\ell=1,2,\dots,r;\]
here $1_{r}$ denotes the $r\times r$ identity matrix. The sum of
the traces of $X,Y,Z$ must then be $2r\omega(m)$, and this can be
written in the equivalent form\[
\sum_{\ell=0}^{r}\ell(\alpha_{\ell}+\beta_{\ell}+\gamma_{\ell})=r\omega(m).\]

As seen in the introduction, the possibility of reductions for the
Schubert intersection problem defined by the sets $I,J,K\subset\{1,2,\dots,n\}$
is tested by calculating an appropriate sum of the indices in these
sets. We are now ready to discuss these sums in full generality. Assume
therefore that $r$ is fixed, $I,J,K\subset\{1,2,\dots,n\}$ and $I',J',K'\subset\{1,2,\dots,n'\}$
are sets of cardinality $r$ such that $c_{IJK}>0$ and $c_{I'J'K'}>0$.
Let us set $\omega=n-r$ and $\omega'=n'-r$. Choose measures $m,m'\in\mathcal{M}_{r}$
such that $\omega(m)=\omega$, $\omega(m')=\omega'$, and $I,J,K$
(resp. $I',J',K'$) are derived from $m$ (resp. $m'$) via (\ref{eq:15}).
We denote by $\alpha_{\ell},\beta_{\ell},\gamma_{\ell}$ (resp. $\alpha'_{\ell},\beta'_{\ell},\gamma'_{\ell})$
the exit densities of $m$ (resp. $m'$). The sum we are interested
in is\[
\mbox{\ensuremath{\Sigma}}_{m'}(m)=\sum_{\ell=1}^{r}(\alpha_{\ell}'i_{\ell}+\beta_{\ell}'j{}_{\ell}+\gamma_{\ell}'k{}_{\ell})-\omega'n.\]
Observe that $\Sigma_{m'}(m)$ depends only on the exit densities
of $m$ and $m'$, and therefore it can be calculated directly from
the sets $I,J,K$ and $I',J',K'$.

The general reduction will proceed as follows. Assume that we want
to solve the Schubert problem associated to a measure $m\in\mathcal{M}_{r}$.
We calculate the sum $\Sigma_{m'}(m)$ for a certain kind of measure
$m'$ (a rigid tree measure in the terminology introduced below).
If this sum is equal to $-p<0$, then one can effectively reduce the
intersection problem to solving first an intersection problem for
a stretched version of $m'$, followed by the intersection problem
for $m-pm'$, for which we have $c_{m-pm'}=c_{m}$; see Theorem \ref{thm:reduction-procedure}.
The problem corresponding to the stretched version of $m'$ can be
solved algorithmically, as seen in \cite{bcdlt}.

Since $n=\omega(m)+r$, we can rewrite\[
\Sigma_{m'}(m)=\sum_{\ell<\ell'}(\alpha_{\ell}\alpha'_{\ell'}+\beta_{\ell}\beta'_{\ell'}+\gamma_{\ell}\gamma'_{\ell'})-\omega(m)\omega(m')+\left[\sum_{\ell=0}^{r}\ell(\alpha_{\ell}'+\beta_{\ell}'+\gamma_{\ell}')-r\omega(m')\right].\]
We have seen earlier that the sum inside the brackets is equal to
zero, and thus\begin{equation}
\Sigma_{m'}(m)=\sum_{\ell<\ell'}(\alpha_{\ell}\alpha'_{\ell'}+\beta_{\ell}\beta'_{\ell'}+\gamma_{\ell}\gamma'_{\ell'})-\omega(m)\omega(m').\label{eq:sigma-no-n}\end{equation}
This formula has several advantages: it does not depend explicitly
on $r$, and by including the branch points of $m$ and $m'$ in a
triangle of a different size we do not alter the sum. More precisely,
if we enlarge the triangle containing the branch points of the measures,
the value of $r$ changes, but the \emph{nonzero }values $\alpha_{\ell},\alpha'_{\ell}$
remain the same, and they appear in the same order, leaving the sum
$\Sigma_{m'}(m)$ unchanged. The arguments in the remainder of the
paper are easier to visualize when all the branch points are contained
in the interior of $\triangle_{r}$, and the reader is free to make
this additional assumption at any point. Another change which does
not affect the value of $\Sigma_{m'}(m)$ is homothety. Denote by
$S$ and $S'$ the supports of $m$ and $m'$, and let $q$ be a positive
integer. It is then possible to define measures $\mu$ and $\mu'$
supported by $qS$ and $qS'$, respectively, and such that the density
of each segment of the form $qe$ is the original density of $e$.
It is obvious that $\Sigma_{\mu'}(\mu)=\Sigma_{m'}(m)$. Taking, for
instance, $q=2$, each small edge in the support of $m$ turns into
two collinear small edges in the support of $\mu$. It is thus possible
to assume that for every small edge \emph{$e$} in the support of
$m$ there is a second, collinear, edge $e'$ which meets $e$ in
a vertex $V$ which is not a branch point. This is a formal way to
perform an operation which is referred to as `breaking an edge in
half' later on. 

The fact that $\omega(m)\omega(m')=\sum_{\ell,\ell'=1}^{r}\alpha_{\ell}\alpha'_{\ell'}$
implies easily that\[
\Sigma_{m'}(m)+\Sigma_{m}(m')=\omega(m)\omega(m')-\sum_{\ell=1}^{r}(\alpha_{\ell}\alpha'_{\ell}+\beta_{\ell}\beta'_{\ell}+\gamma_{\ell}\gamma'_{\ell}).\]
In particular, when $m=m'$ we have\[
\Sigma_{m}(m)=\frac{1}{2}\left[\omega(m)^{2}-\sum_{\ell=1}^{r}(\alpha_{\ell}^{2}+\beta_{\ell}^{2}+\gamma_{\ell}^{2})\right],\]
a formula requiring fewer multiplications.

\section{Trees and Measures}

Some measures $m\in\mathcal{M}_{r}$ have an underlying tree structure
which we describe next. We start with a special class of planar trees.
We consider trees embedded in the usual Euclidean plane such that
\begin{enumerate}
\item each edge of the tree is a straight line segment of unit length, 
\item each vertex has order 2 or 3, and
\item there are only finitely many vertices of order 3.
\end{enumerate}
These conditions imply that the tree is infinite, but it has a finite
number of \emph{ends}. These are sequences of vertices of the form
$V_{0}V_{1}\cdots$ such that $V_{0}$ has order 3, $V_{j}$ has order
2 for $j\ge1$, and $V_{j}V_{j+1}$ is an edge for each $j\ge0$.
We will require one more condition on our trees.
\begin{enumerate}
\item [(4)]The shortest path joining two different ends contains an odd
number of vertices of order 3.
\end{enumerate}
All the trees we use will satisfy these four properties, and therefore
we will not introduce a special name for this particular species.
An \emph{immersion }of a tree $T\subset\mathbb{R}^{2}$ is simply
a continuous map $\varphi:T\to\mathbb{R}^{2}$ which 
\begin{itemize}
\item is  isometric on each edge,
\item if $VA$ and $VB$ are the two edges meeting at a vertex of order
2, then $2\varphi(V)=\varphi(A)+\varphi(B)$, and
\item if $VA,VB,VC$ are the three edges meeting at a vertex of order 3,
then $3\varphi(V)=\varphi(A)+\varphi(B)+\varphi(C)$, and the restriction
of $\varphi$ to $VA\cup VB\cup VC$ preserves the orientation.
\end{itemize}
It is clear that each tree has a unique immersion up to rigid motions.
Immersions are generally not one-to-one. A tree $T$ is endowed with
arclength measure. Given an immersion $\varphi$ of $T$, we consider
the push-forward $m_{\varphi}$ of this measure. Thus, if we arrange
our immersion such that $\varphi(T)$ is contained in the small edges
of the triangular lattice determined by the vectors $u,v,w$, then
$m$ assigns to each edge a density equal to the number of its preimages
in $T$. The resulting measure clearly satisfies the balance condition
(\ref{eq:balance}) at all vertices. Condition (4) implies that we
can arrange $\varphi$ so that $m_{\varphi}\in\mathcal{M}_{r}$ provided
that $r$ is sufficiently large (so that $\triangle_{r}$ contains
$\varphi(V)$ whenever $V$ is a vertex of order 3 of $T$). A measure
$m\in\mathcal{M}_{r}$ will be called a \emph{tree measure} if $m=m_{\varphi}$
for some immersion $\varphi$ of a tree. The following illustration
shows a tree, and the range of one of its immersions. The arrows indicates
ends of the tree, and the asterisk indicates where one of these ends
is mapped by the immersion.

\begin{center}
\includegraphics[scale=0.7]{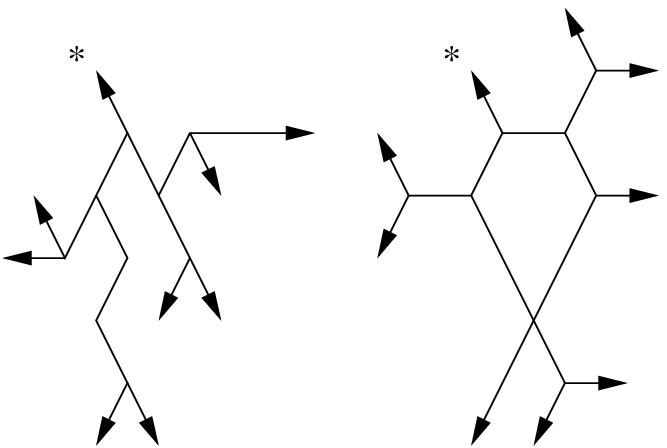}
\par\end{center}

\noindent In the second illustration, some edges of the immersion
have multiplicity two (i.e., they have two preimages under the corresponding
immersion). They are represented by thicker lines.

\begin{center}
\includegraphics[scale=0.7]{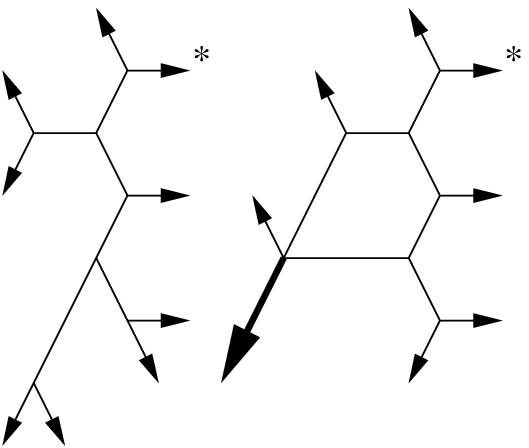}
\par\end{center}

\noindent Here is one more figure illustrating the fact that a tree
measure need not be extremal.

\begin{center}
\includegraphics[scale=0.7]{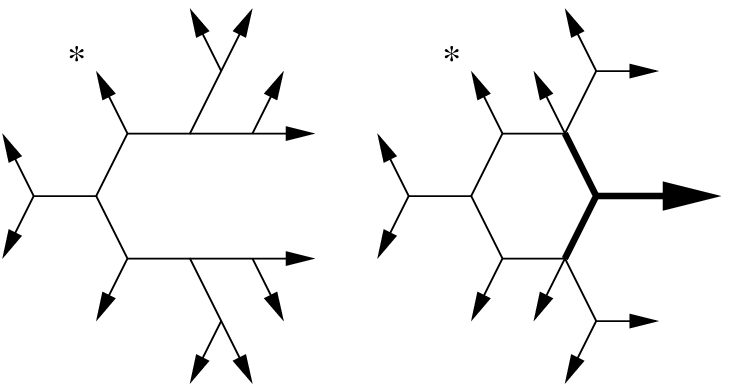}
\par\end{center}

\noindent In this case, the measure $m_{\varphi}$ has  two summands
with unit densities; the support of one of them is pictured below.

\begin{center}
\includegraphics[scale=0.7]{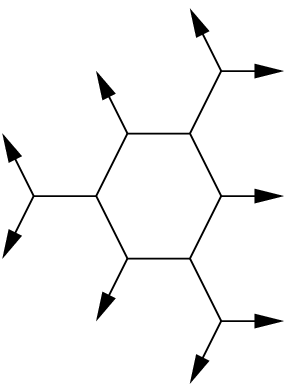}
\par\end{center}

If $m\in\mathcal{M}_{r}$ is a tree measure, it is fairly easy to
see that the number of ends of the corresponding tree $T$ is $3\omega(m)$.
We will write $\omega(T)=\omega(m)$. For the trees above, the value
of $\omega(T)$ is 3 or 4.
\begin{thm}
\label{thm:rigid-skeletons-are-trees}Assume that $m\in\mathcal{M}_{r}$
is a rigid extremal measure. Then there exists a tree measure $m'\in\mathcal{M}_{r}$
such that $m=cm'$ for some constant $c>0$.\end{thm}
\begin{proof}
Assume, more generaly, that $m\in\mathcal{M}_{r}$ is a rigid measure.
Given two adjacent small edges $AB,BC$ in the support of $m$, we
write $AB\to_{m}BC$ if either 
\begin{enumerate}
\item [(a)] $A,B,C$ are collinear and one of the edges $BX$ such that
$\varangle XBC=60^{\circ}$ satisfies $m(BX)=0$, or 
\item [(b)] $\varangle ABC=120^{\circ}$, and the edge $BX$ opposite $AB$
satisfies $m(BX)=0$. 
\end{enumerate}
Given an edge $e=AB$, there exist at most two edges $f$ adjacent
to $B$ such that $e\to_{m}f$. More generally, if $e,f$ are two
small edges, we write $e\Rightarrow_{m}f$ if either $e=f$, or \[
e=e_{1}\to_{m}e_{2}\to_{m}\cdots\to_{m}e_{k}=f\]
for some chain $\gamma=\{e_{1},e_{2},\dots,e_{k}\}$, $e_{j}=X_{j-1}X_{j}$,
of small edges. This relation is called \emph{descendance}, and it
was proved in \cite{bcdlt} that each edge in the support of $m$
is the descendant of a minimal (or \emph{root}) edge contained in
$\triangle_{r}$. Moreover, the descendants of a root edge form the
support of an extremal measure. Here minimality is defined up to the
equivalence relation $e\Leftrightarrow_{m}f$ if $e\Rightarrow_{m}f$
and $f\Rightarrow_{m}e$. A chain $\gamma$ as above is called a \emph{descendance
path} from $e$ to $f$. 

Assume now that $m$ is extremal and $e$ is a root edge for $m$
contained in $\triangle_{r}$. Dividing $m$ by $c=m(e)$, we may
assume that $m(e)=1$. If $f$ is any edge in the support of $m$,
$m(f)$ equals the number of descendance paths from $e$ to $f$ (cf.
\cite{bcdlt}). Note that $m$ may have several (often, infinitely
many) root edges $f$; they are characterized by the equality $m(f)=1$. 

The construction of the required tree $T$ is somewhat analogous to
the construction of a universal covering space. Abstractly, the vertices
of $T$ are sequences $X_{0}X_{1}\cdots X_{n}$ such that either $n=0$
and $X_{0}$ is an endpoint of $e$, or $n\ge1$ and $\gamma=\{X_{0}X_{1},X_{1}X_{2},\dots,X_{n-1}X_{n}\}$
is a descendance path from $e$. The vertices $X_{0},X_{1}$ are identified
with $X_{1}X_{0},X_{0}X_{1}$, respectively, if $X_{0}$ and $X_{1}$
are the endpoints of $e$. Two vertices of the form $X_{0}X_{1}\cdots X_{n}$,
$X_{0}X_{1}\cdots X_{n}X_{n+1}$ are joined by an edge. Assigning
unit length to the edges of $T$, there is a map $\varphi:T\to\mathbb{R}^{2}$
which sends a vertex $X_{0}X_{1}\cdots X_{n}$ to $X_{n}$. We embed
the tree $T$ into the plane in such a way that this map $\varphi$
preserves orientation at each triple vertex of $T$. It should be
clear now that $m=m_{\varphi}$.
\end{proof}
Let $\varphi$ be the immersion of $T$ described in the preceding
proof, and let $e$ be an edge of $T$ such that $\varphi(e)$ is
a root edge for the measure $m$. We can orient all other edges of
$T$ \emph{away from} $e$. It was shown in \cite{bcdlt} that the
map $\varphi$ has the following additional property: if $g$ and
$h$ are two edges such that $\varphi(g)=\varphi(h)$, then $\varphi$
induces the same orientation on this common image. In other words,
the edges in the support of $m$, other than $\varphi(e)$, can be
consistently oriented in the direction of a descendance path from
$\varphi(e)$. The following lemma is also proved in \cite{bcdlt}
(see the discussion following Theorem 3.5 in \cite{bcdlt}).
\begin{lem}
\label{lem:local-structure-of-skeletons}Let $m$ be a rigid extremal
measure, and orient the edges in its support away from a fixed root
edge. Each lattice point meets at most four edges in the support of
$m$, and the possible positions of these edges, including their orientations,
are as follows

\begin{center}\includegraphics{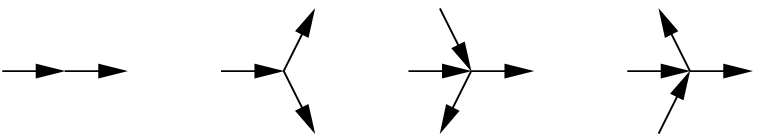}\end{center}

\noindent up to rotations.
\end{lem}
In order to study the sums $\Sigma_{m'}(m)$, we will also need some
maps which are closely related to immersions, but are discontinuous.
Assume that $T$ is a tree, and $\varphi$ is an immersion of $T$
such that the induced measure is in $\mathcal{M}_{r}$ for some $r$.
Denote by $T_{\circ}$ the set of points in $T$ which are not vertices.
A function $\psi:T_{\circ}\to\mathbb{R}^{2}$ will be called a \emph{fractured
immersion} if 
\begin{enumerate}
\item the range of $\psi$ is contained in the small edges of the triangular
lattice determined by $u,v,w$,
\item there is an immersion $\varphi$ of $T$ such that $\psi(t)-\varphi(t)$
is constant on the interior of every edge, and
\item $\psi$ extends continuously to all except finitely many vertices
of $T$.
\end{enumerate}
Let $\psi$ be a fractured immersion of a tree $T$. We will associate
to each vertex $V$ of $T$ an integer $\delta_{\psi}(V)$ which measures
how badly fractured $\psi$ is at $V$. If $\psi$ extends continuously
to the point $V$ we set $\delta_{\psi}(V)=0$. Assume next that the
order of $V$ is 2 and the two edges $AV,VB$ are mapped to $A'V',V''B'$,
respectively, with $V'\ne V''$. We will set $\delta_{\psi}(V)=q$
if the point $V''$ lies $q$ lattice units \emph{to the left} of
the line joining $A'$ and $V'$, where this line is oriented so that
$A'V'$ points toward $V'$. Note that $V''$ could be to the right
of this line, in which case $q<0$, and $V''$ (as well as $B'$)
could be on this line, in which case $q=0$. Finally, let $V$ be
a vertex of order 3, assume that the three edges $AV,BV,CV$ are mapped
to $A'V',B'V'',C'V'''$, and note that these three segments still
form $120^{\circ}$ angles. If the lines containing these three segments
are concurrent, we set $\delta_{\psi}(V)=0$. Otherwise, these three
lines form an equilateral triangle $\triangle$ with sidelength $q$.
Orient the sides of this triangle so that the segments $A'V',B'V'',C'V'''$
point toward $V',V'',V'''$, respectively. If the boundary of $\triangle$
is oriented clockwise, set $\delta_{\psi}(V)=-q$, and in the contrary
case set $\delta_{\psi}(V)=q$. The following figures shows three
cases in which the values of $\delta_{\psi}(V)$ are $0,-2$ and $1$.
The dotted lines represent small edges.

\begin{center}
\includegraphics{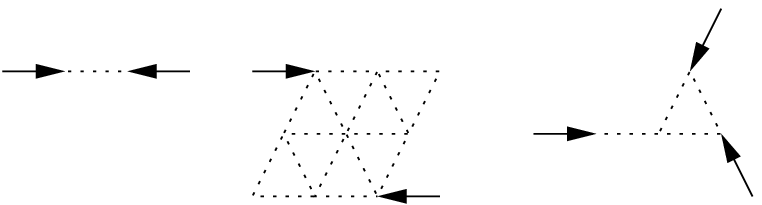}
\par\end{center}

\noindent The orientations indicated above are used exclusively for
the calculation of the numbers $\delta_{\psi}(V)$. In the proofs
below we will need to orient all the edges of a tree $T$ (not just
the ones adjacent to $V$), and this will generally be the orientation
\emph{away} from a fixed vertex or edge.

In the following statement, the segment $A_{0}X_{0}$ is deemed to
exit $\triangle_{r}$ at the point $A_{0}$, rather than $C_{r}$,
while $C_{r}Z_{r}$ is deemed to exit at $C_{r}$. Of course, this
issue does not arise when the corners of $\triangle_{r}$ are not
exit points, and this can be achieved by enlarging the triangle.
\begin{thm}
\label{thm:sum-of-ends-of-broken-tree}Let $\psi$ be a fractured
immersion of a tree $T$ such that all the limits of $\psi$ at discontinuity
points are contained in $\triangle_{r}$. For each end $E$ of $T$,
denote by $\ell(E)$ the rank of the exit point of $\psi(E)$ from
$\triangle_{r}$. In other words, $\ell(E)=\ell$ if the closure of
$\psi(E)$ intersects $\partial\triangle_{r}$ in $A_{\ell},B_{\ell},$
or $C_{\ell}$. Then we have\[
\sum_{{\rm all\,\, ends}\: E\:{\rm of}\: T}\ell(E)=r\omega(T)+\sum_{{\rm all\,\, vertices}\: V\:{\rm of}\: T}\delta_{\psi}(V).\]
\end{thm}
\begin{proof}
We proceed by induction on the number of vertices where $\psi$ does
not extend continuously. When this number is equal to zero, $g$ is
an immersion, and the sum in the left hand side is nothing but\[
\sum_{\ell}\ell(\alpha_{\ell}+\beta_{\ell}+\gamma_{\ell})=r\omega(m),\]
where $\alpha_{\ell},\beta_{\ell},\gamma_{\ell}$ are the exit densities
of the corresponding measure $m$. This is precisely the desired identity
because $\omega(m)=\omega(T).$ Assume then that the theorem has been
proved for all fractured immersions with fewer discontinuity points
than $\psi$, and there exists at least one vertex $V$ where $\psi$
does not extend continuously. Consider first the case when $V$ is
of order $2$, and the vertices $AV,VB$ are mapped by $\psi$ to
$A'V',V''B'$, which we will assume to be horizontal for definiteness.
By transposing the points $A,B$, we can also assume that $A'$ is
to the left of $V'$ and $B'$ is to the right of $V''$. There is
then a point $A_{\ell}$ such that the segment $A_{\ell}V''$ is horizontal;
denote by $a$ its length. Similarly, there is a point $C_{k}$ such
that the segment $V'C_{k}$ is horizontal. The definition of $\delta$
implies that\begin{equation}
\ell+k+\delta_{\psi}(V)=r.\label{eq:l+k}\end{equation}
We now form two trees in the following way. Cut the tree $T$ at the
point $V$, and add to the part containing $AV$ an end $VV_{1}V_{2}\cdots$,
thus forming a tree $T'$. Analogously, add to the part containing
$BV$ a path $VW_{1}W_{2}\cdots W_{a}$, where $W_{1},W_{2},\dots,W_{a-1}$
have order 2, and two ends meeting at $W_{a}$, thus forming a tree
$T''$. The map $\psi$ gives rise to two fractured immersions $\psi'$
and $\psi''$ of $T'$ and $T''$ as follows: $\psi'(VV_{1}V_{2}\cdots)$
is the half line starting with $A'V'$, $\psi''($$VW_{1}W_{2}\cdots W_{a})=V''A_{\ell}$,
and the two ends meeting at $W_{a}$ are mapped onto the two half
lines starting at $A_{\ell}$ and pointing left. It is clear that
$\psi'$ and $\psi''$ have fewer vertices of discontinuity than $\psi$,
and therefore the desired formula is true for $\psi'$and $\psi''$.
It is clear that\[
\sum_{\text{vertices }W\text{ of }T}\delta_{\psi}(W)=\sum_{\text{vertices }V'\text{ of }T'}\delta_{\psi'}(V')+\sum_{\text{vertices }V''\text{ of }T''}\delta_{\psi''}(V'')+\delta_{\psi}(V),\]
while

\[
\sum_{\text{all ends }E'\text{ of }T'}\ell(E')+\sum_{\text{all ends }E''\text{ of }T''}\ell(E'')=\sum_{\text{all ends }E\text{ of }T}\ell(E)+\ell+k.\]
 The desired equality follows then from (\ref{eq:l+k}) because $\omega(T')+\omega(T'')=\omega(T)+1$.
The solid arrows in the following illustration are the oriented segments
$A'V'$ and $B'V''$, while the dashed lines indicate where the additional
edges in $T'$ and $T''$ are mapped,. Their exit points from $\triangle_{r}$
are $A_{\ell}$, $B_{0}$ and $C_{k}$.

\begin{center}\includegraphics{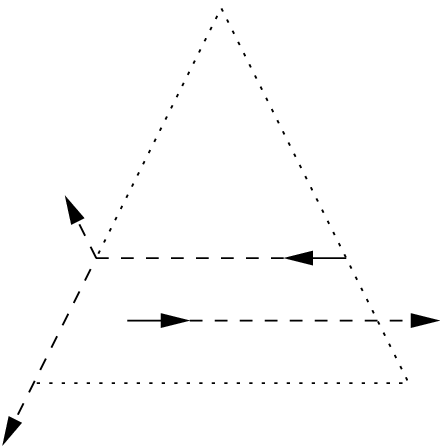}\end{center}

Consider next the case that $V$ is of order three, and the three
edges $AV,BV,CV$ of $T$ are mapped to $A'V',B'V'',C'V'''$. Assume
that $A,B,C$ are arranged clockwise around $V$. A cyclic permutation
allows us to assume that $A'V'$ is horizontal, and we must consider
the two cases where $A'$ is to the left or to the right of $V'$.
These two situations are illustrated below.

\begin{center}\includegraphics{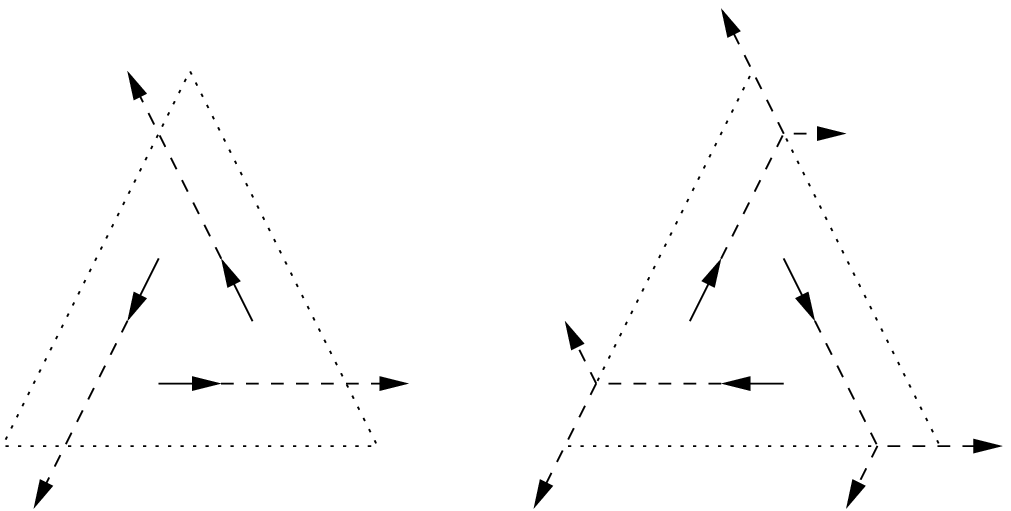}\end{center}

\noindent Assume first that $A'$ is on the left. The half lines $A'V'$,
$B'V''$, $C'V'''$ exit $\triangle_{r}$ at points $C_{k},B_{\ell},A_{p}$,
respectively. As in the preceding proof, we cut $T$ at the point
$V$, and form three trees $T',T'',T'''$ by attaching to the part
of $T$ which contains $A,B,C,$ respectively, an end attached at
$V$. The map $\psi$ gives rise to three fractured immersions $\psi',\psi'',\psi'''$
of these trees. For instance, $\psi'$ maps the additional end at
$V$ to the half line starting with $V'C_{k}$. Moreover, the new
fractured immersions have fewer discontinuity points than $\psi$,
and therefore the inductive hypothesis applies to them. As in the
preceding case, we have $\omega(T')+\omega(T'')+\omega(T''')=\omega(T)+1$,
and\[
k+\ell+p+\delta_{\psi}(V)=r.\]
 The desired formula follows now easily. Finally, consider the case
in which $A'$ is to the right of $V'$. In this case, the half lines
$A'V',B'V'',C'V'''$ exit $\triangle_{r}$ at points $A_{\ell},B_{k},C_{p}$,
respectively, and the trees $T',T'',T'''$ must be constructed by
attaching at $V$ a few edges followed by two ends. In this case we
have $\omega(T')+\omega(T'')+\omega(T''')=\omega(T)+2$ and the reader
can verify easily that $k+\ell+p+\delta_{\psi}(V)=2r$. The conclusion
follows as before.
\end{proof}

\section{Inflations and Fractured Immersions}

We recall from \cite{KTW} (see also \cite{bcdlt}) that every measure
$\nu\in\mathcal{M}_{r}$ has an associated puzzle obtained by inflating
$\nu$. The \emph{inflation} of $\nu$ is defined as follows. Cut
the plane along the edges in the support of $\nu$ to obtain a collection
of \emph{puzzle} pieces, and translate these pieces away from each
other in the following way: the parallelogram formed by the two translates
of a side $AB$ of a white puzzle piece has two sides of length equal
to the density of $\nu$ on $AB$ and $60^{\circ}$ clockwise from
$AB$. The balance condition (\ref{eq:balance}) implies that the
original puzzle pieces and these parallelograms fit together, and
leave a space corresponding to each branch point in the support of
$\nu$. Here is an illustration of the process with $r=3$; the thinner
lines in the support of the measure have density one, and the thicker
ones density 2. The original pieces of the triangle $\triangle_{r}$
are white, the added parallelogram pieces are dark gray, and the branch
points become light gray pieces. Each light gray piece has as many
sides as there are branches at the original branch point.

\begin{center}
\includegraphics{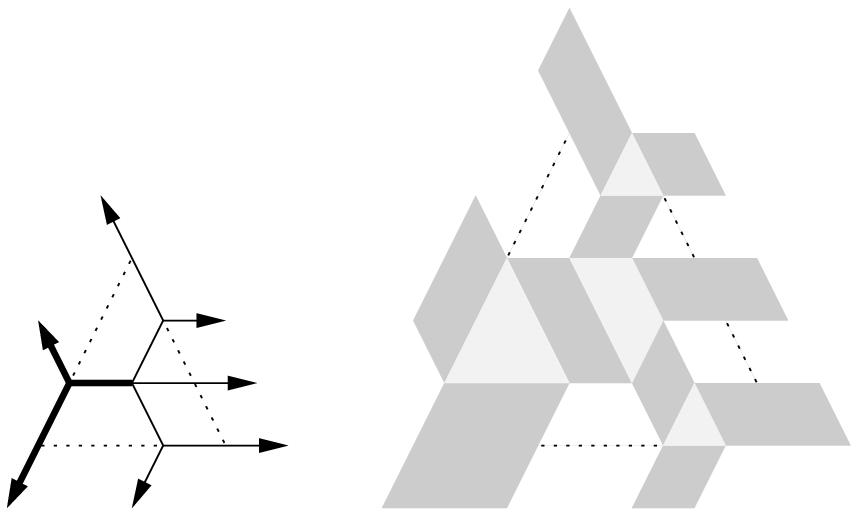}
\par\end{center}

\noindent The dotted lines indicating the boundary of $\triangle_{r}$
have been translated so that they now outline a triangle with sides
$r+\omega(\nu)$, which we may assume is precisely $\triangle_{r+\omega(\nu)}$.
The decomposition of this triangle into white, dark gray, and light
gray pieces is the puzzle associated to $\nu$. The white regions
in the puzzle are called `zero regions', and the light gray ones `one
regions', and the dark gray parallelograms `0-1 regions' in \cite{KTW}.

The main use of inflations will be to produce fractured immersions
from a given immersion of a tree $T$. Assume indeed that $T$ is
a tree, $\varphi$ is an immersion of $T$ such that the corresponding
measure $m'=m_{\varphi}$ is in $\mathcal{M}_{r}$, and let $\nu\in\mathcal{M}_{r}$
be another measure. Assume that each edge of $T$ has been given an
orientation, and that all the edges belonging to an end of $T$ have
been oriented outward (i.e., towards the infinite part of that end).
For each edge $e$ in $T$ such that $\varphi(e)$ is in the support
of $\nu$, we attach $\varphi(e)$ to the white puzzle piece \emph{on
the right of} $\varphi(e)$ when $\varphi(e)$ is given the orientation
induced by the orientation of $e$. For edges $e$ with $\nu(e)=0$,
$\varphi(e)$ is contained in a white puzzle piece, and it moves along
with that piece. If we denote now by $\psi(e)$ the translate of $\varphi(e)$
in the puzzle construction, we obviously obtain a fractured immersion.
The following figure illustrates the process as applied to a measure
$m'=m_{\varphi}$ whose support is pictured below, and $\nu$ is the
measure whose inflation was depicted in the preceding figure. We have
oriented all the edges away from the branch point inside $\triangle_{3}$,
and completed the outline of $\triangle_{6}$.

\begin{center}
\includegraphics{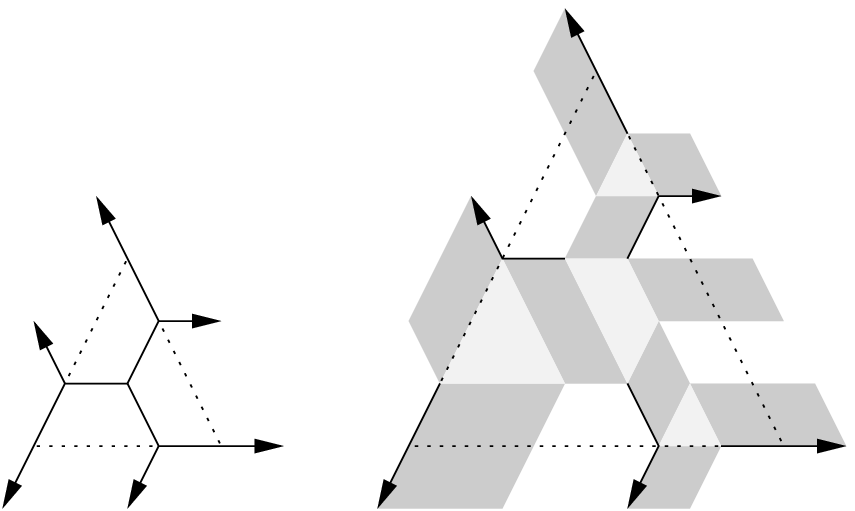}
\par\end{center}

Note that all the fractures of $\psi$ are contained in $\triangle_{r+\omega(\nu)}$,
and therefore the formula in Theorem \ref{thm:sum-of-ends-of-broken-tree}
applies. Let $\alpha_{j,}\beta_{j},\gamma_{j}$ be the exit densities
of $\nu$, and let $\alpha'_{j},\beta'_{j},\gamma'_{j}$ be the exit
densities of $m'$. Then it is easy to see that\begin{eqnarray*}
\sum_{\text{all ends }E\text{ of }T}\ell(E) & = & \sum_{\ell=0}^{r}\left[\alpha'_{\ell}\left(\ell+\sum_{k<\ell}\alpha_{k}\right)+\beta_{\ell}'\left(\ell+\sum_{k<\ell}\beta_{k}\right)+\gamma_{\ell}'\left(\ell+\sum_{k<\ell}\gamma_{k}\right)\right]\\
 & = & r\omega(m')+\sum_{k<\ell}(\alpha_{\ell}'\alpha_{k}+\beta_{\ell}'\beta_{k}+\gamma_{\ell}'\gamma_{k}).\end{eqnarray*}
Indeed, this follows from the fact that an end $E$ such that $\varphi(E)$
exits at $A_{\ell}$ is translated to $\psi(E)$ which exits at $A_{\ell+\alpha{}_{0}+\cdots+\alpha{}_{\ell-1}}$.
\begin{lem}
\label{lem:sigma=00003Dsum-of-delta}With the notation above, we have\[
\Sigma_{m'}(\nu)=\sum_{{\rm all\,\, vertices}\: V\:{\rm of}\: T}\delta_{\psi}(V).\]
\end{lem}
\begin{proof}
Theorem \ref{thm:sum-of-ends-of-broken-tree} yields\[
\sum_{\text{all ends }E\text{ of }T}\ell(E)=(r+\omega(\nu))\omega(m')+\sum_{{\rm all\,\, vertices}\: V\:{\rm of}\: T}\delta_{\psi}(V).\]
Combining this with the identity preceding the statement, we obtain\[
\sum_{\text{all vertices }V\text{ of }T}\delta_{\psi}(V)=\sum_{k<\ell}(\alpha_{\ell}'\alpha{}_{k}+\beta_{\ell}'\beta{}_{k}+\gamma_{\ell}'\gamma{}_{k})-\omega(m')\omega(\nu),\]
and this is precisely the formula (\ref{eq:sigma-no-n}) for $\Sigma_{m'}(\nu)$.\end{proof}
\begin{thm}
\label{thm:sigma(m,m)}Assume that $\nu,m'\in\mathcal{M}_{r}$, and
$m'$ is a tree measure.
\begin{enumerate}
\item If the support of $m'$ is not contained in the support of $\nu$,
then $\Sigma_{m'}(\nu)\ge0$.
\item If $m'$ is not rigid, we also have $\Sigma_{m'}(m')\ge0$.
\item If $m'$ is an extremal rigid measure assigning unit density to its
root edges, we have $\Sigma_{m'}(m')=-1$.
\end{enumerate}
\end{thm}
\begin{proof}
Let $\varphi$ be an immersion of a tree $T$ such that $m'=m_{\varphi}$.
To prove (1), fix an edge $e_{0}$ such that $\varphi(e_{0})$ is
not contained in the support of $\nu$, and orient all the other edges
of $T$ \emph{away from} $e_{0}$. Construct a fractured immersion
$\psi$ using the above construction associated with the inflation
of $\nu$. It is easy to verify that in this case we have $\delta_{\psi}(V)\ge0$
for every vertex $V$ of $T$. Indeed, $\delta_{\psi}(V)$ can be
calculated explicitly in terms of the values of $\nu$ on one of the
edges adjacent to $\varphi(V)$. To see this, assume first that $V$
is of order two, $AV$ and $VB$ are the two adjacent edges, and they
are mapped by $\varphi$ to $A'V'$ and $V'B'$. These two edges are
shown below, with the arrows indicating their orientation, and the
dotted extensions are drawn to indicate the value of $\delta_{\psi}(V)$. 

\begin{singlespace}
\begin{center}\includegraphics{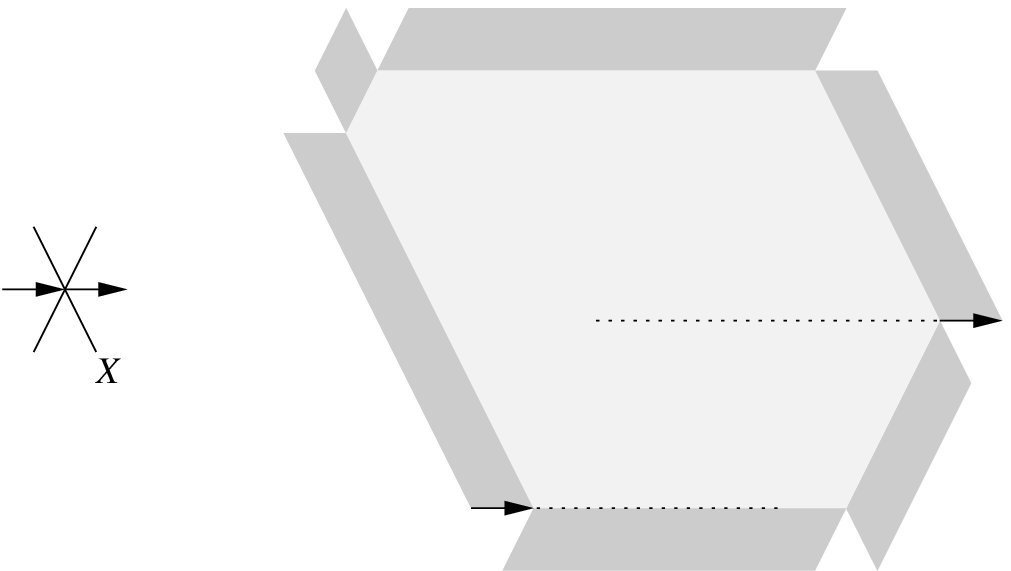}\end{center}
\end{singlespace}

\noindent Clearly, we have $\delta_{\psi}(V)=\nu(V'X)$, with $X$
as in the figure, i.e. on the right side of $A'V'$, and $\varangle XV'B'=60^{\circ}.$
If $V$ has order three, let $AV,BV,CV$ be the three adjacent edges,
with $AV$ oriented toward $V$. Assume that $\varphi(AV)=A'V'$,
and $X$ is symmetric to $A'$ relative to $V'$. We have again $\delta_{\psi}(V)=\nu(V'X)$.

\begin{center}\includegraphics{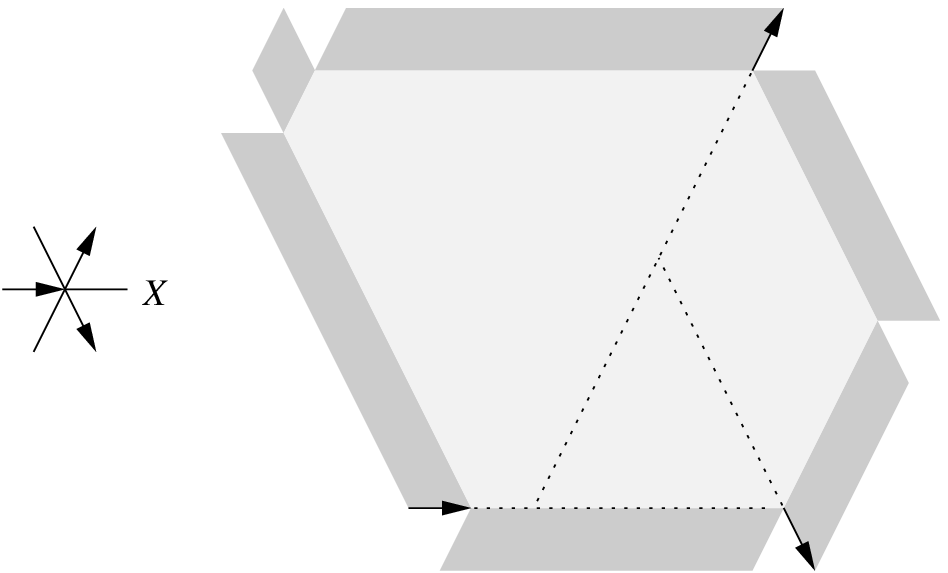}\end{center}

\noindent Assertion (1) follows now from Lemma \ref{lem:sigma=00003Dsum-of-delta}.
(In both illustrations we assumed that $\nu$ assigns nonzero densities
to all six edges adjacent to $V'$. More precisely, these densities
were taken to be $1,7,4,3,5$ and $6$ in clockwise order.)

Assume next that $m'$ is not rigid, and choose a different measure
$\nu$ with the same exit densities. Then $\nu$ can be written as
a sum of distinct extreme measures, say $\nu=\sum_{j}m_{j}$. If the
support of $m_{j}$ contains the support of $m'$, then $m_{j}$ is
a positive multiple of $m'$ by extremality. Thus there is at most
one $j$ such that the support of $m_{j}$ contains the support of
$m'$. Assume for definiteness that $m_{1}=\kappa m'$, where $0\le\kappa<1$.
Then part (1) of the theorem implies that $\Sigma_{m'}(m_{j})\ge0$
for $j\ne1$, hence \[
\Sigma_{m'}(m')=\Sigma_{m'}(\nu)=\sum_{j}\Sigma_{m'}(m_{j})\ge\Sigma_{m'}(m_{1})=\kappa\Sigma_{m'}(m'),\]
and therefore $\Sigma_{m'}(m')\ge0$, as claimed. 

Finally, assume that $m'$ is rigid, and choose an edge $e_{0}$ such
that $\varphi(e_{0})$ is a root edge for $\nu=m'$ contained in $\triangle_{r}$.
Orient the other edges $T$ away from $e_{0}$, and also give $e_{0}$
some orientation, say it is oriented away from one of its endpoints
$V_{0}$. In this case we have $\delta_{\psi}(V_{0})=-1$ and $\delta_{\psi}(V)=0$
for all other vertices. To verify this fact one must observe that
in the pictures above we must have $m'(V'X)=0$ because of the rigidity
of $m'$. This follows from Lemma \ref{lem:local-structure-of-skeletons}.
The only exception is the orientation at the point $V_{0}$ which
produces a nonzero $\delta_{\psi}(V_{0})$. To calculate the value
of $\delta_{\psi}(V_{0})$, we will further assume that $V_{0}$ is
a vertex of order 2 and both edges $A_{0}V_{0,}V_{0}B_{0}$ adjacent
to $V_{0}$ are mapped by $\varphi$ to root edges of $m'$. This
can be achieved by applying a homothety, as seen in the introduction.
Assuming, for instance, that $\varphi(A_{0})=A$, $\varphi(B_{0})=B$
and $\varphi(V_{0})=V$, we have $m'(AV)=m'(VB)=1$. If we orient
$A_{0}V_{0}$ and $B_{0}V_{0}$ away from $V_{0}$, the inflation
process looks as follows:

\begin{center}\includegraphics{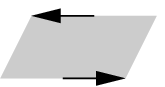}\end{center}

\noindent The width of the dark gray parallelogram is $m'(AV)=1$,
hence $\delta_{\psi}(V_{0})=-1$. The theorem follows.\end{proof}
\begin{cor}
\label{cor:m-pm}Assume that $m,m'\in\mathcal{M}_{r}$ and $m'$ is
a rigid extreme tree measure. If $\Sigma_{m'}(m)=-p<0$ then $pm'\le m$
and $c_{m-pm'}=c_{m}$.\end{cor}
\begin{proof}
Let $s$ be the largest number such that $sm'\le m$. Then the support
of $m-sm'$ does not contain the support of $m'$, and therefore $\Sigma_{m'}(m-sm')\ge0$
by Theorem \ref{thm:sigma(m,m)}(1). Thus \[
-p=\Sigma_{m'}(m)=\Sigma_{m'}(m-sm')+s\Sigma_{m'}(m')\ge-s,\]
so that $s\ge p$. If $m''$ is any other measure with the same exit
densities as $m$, it follows that $pm'\le m''$ as well, and the
exit densities for $m-pm'$ and $m''-pm'$ are the same. This yields
a bijection $m''\leftrightarrow m''-pm'$ between measures with the
exit densities of $m$ and measures with the exit densities of $m-pm'$.
\end{proof}
Corollary \ref{cor:m-pm} extends \cite[Proposition 3.10]{CoDy-reduction}
which, in our terminology, states that $c_{m-m'}=c_{m}$ if $m'$
is a tree measure with $\omega(m')=1$, and $\Sigma_{m'}(m)<0$. That
result was stated in terms of the sets $I,J,K$, and the proof proceeds
through a very explicit construction of Littlewood-Richardson tableaux.

We can now give a general method for the construction of rigid measures,
thus completing \cite[Theorem 3.8]{bcdlt}. First, we need to review
that result. Let $m\in\mathcal{M}_{r}$ be a rigid measure, and let
$m_{1},m_{2}\in\mathcal{M}_{r}$ be two tree measures with support
contained in the support of $m$. The relation $m_{1}\prec_{0}m_{2}$
was defined in \cite{bcdlt} as follows: there exist four small edges
$AX,XB,CX$ and $XD$ such that
\begin{enumerate}
\item $AX$ and $XB$ are collinear edges in the support of $m_{1}$,
\item $CX$ and $XD$ are collinear edges in the support of $m_{2}$, and
\item $XB$ is $60^{\circ}$ clockwise from $XD$.
\end{enumerate}
It was shown in \cite{bcdlt} that `$\prec_{0}$' can be extended
to an order relation on the set of extremal rigid measures with support
contained in the support of $m$. As noted earlier, each extremal
rigid measure is a positive multiple of a tree measure. The following
result allows us to extend `$\prec_{0}$' to the collection of all
rigid tree measures; this extension is no longer contained in an order
relation.
\begin{lem}
\label{lem:precedence-vs-Sigma}Let $m$ be a rigid measure, and let
$m_{1},m_{2}$ be extremal measures with support contained in the
support of $m$. We have $m_{1}\prec_{0}m_{2}$ if and only if $\Sigma_{m_{2}}(m_{1})>0$. \end{lem}
\begin{proof}
Observe that $m_{1}$ and $m_{2}$ are also rigid. Let $\varphi$
be an immersion of some tree $T$ such that $m_{2}=m_{\varphi}$;
such an immersion exists by Theorem \ref{thm:rigid-skeletons-are-trees}.
Orient all the edges of $T$ away from some edge $e_{0}$ such that
$\varphi(e_{0})$ is a root edge for $m_{2}$ not contained in the
support of $m_{1}$. Assume first that $m_{1}\prec_{0}m_{2}$, and
the small edges $AX,XB,CX,XD$ satisfy conditions (1-3) above. We
may assume that $CX=\varphi(e_{1}),XD=\varphi(e_{2})$, where $e_{1}$
and $e_{2}$ are adjacent edges, and $e_{1}$ is oriented toward $e_{2}$.
The proof of Theorem \ref{thm:sigma(m,m)} implies that $\Sigma_{m_{2}}(m_{1})\ge m_{1}(XB)>0$. 

Conversely, assume that $\Sigma_{m_{2}}(m_{1})>0$. Let $e_{1},e_{2},e_{3}$
be three edges of $T$ adjacent to a vertex $V$, and assume that
$e_{1}$ is oriented toward $V$. These edges are mapped by $\varphi$
to $A_{j}X$, $j=1,2,3$, and we must have $A_{1}X\to_{m}XA_{2}$
and $A_{1}X\to_{m}XA_{3}$. It follows that the edge $XB$ opposite
$A_{1}X$ satisfies $m(XB)=0$, and therefore $m_{1}(XB)=0$, so that
this vertex $V$ contributes nothing to $\Sigma_{m_{2}}(m_{1})$.
We conclude that there must exist some vertex $V$ of order 2 which
contributes to $\Sigma_{m_{2}}(m_{1})$. Let $e_{1},e_{2}$ be the
two edges adjacent to $V$, and assume that $e_{1}$ is oriented toward
$V$. Then $\varphi$ maps these two edges to collinear edges $CX,XD$
so that $CX\to_{m}XD$. The fact that $V$ contributes to $\Sigma_{m_{2}}(m_{1})$
means simply that the edge $XB$ which is $60^{\circ}$ clockwise
from $XD$ is in the support of $m_{1}$. We claim that the edge $AX$
opposite $XB$ is also in the support of $m_{1}$. Indeed, the fact
that $CX\to_{m}XD$ implies that the edge $XB'$ which is $60^{\circ}$
counterclockwise from $XD$ is not in the support of $m$, hence not
in the support of $m_{1}$. The balance condition for $m_{1}$ implies
that $m_{1}(AX)>0$. Thus the vertices $AX,XB,CX,XD$ witness the
fact that $m_{1}\prec_{0}m_{2}$.
\end{proof}
Corollary $3.6$ of \cite{bcdlt} allows us to write any rigid measure
$m\in\mathcal{M}_{r}$ under the form \[
m=\sum_{j=1}^{n}p_{j}m_{j},\]
where $p_{j}>0$, and the $m_{j}$ are distinct extremal tree measures.
Moreover, Theorem 3.8 of that paper allows us to arrange the terms
of this sum in such a way that $m_{i}\prec_{0}m_{j}$ implies that
$i\le j$. According to Lemma \ref{lem:precedence-vs-Sigma}, $m_{i}\prec_{0}m_{j}$
is equivalent to $\Sigma_{m_{j}}(m_{i})>0$ for these measures, Thus,
the following result can be viewed as a converse of \cite[Corollary 3.6]{bcdlt}.
\begin{cor}
\label{cor:rigid=00003DnoCycles}Let $m_{1},m_{2},\dots,m_{n}$ be
extremal, rigid tree measures such that $\Sigma_{m_{i}}(m_{j})=0$
for $i<j$. For every $p_{1},p_{2},\dots,p_{n}>0$, the measure $m=\sum_{j=1}^{n}p_{j}m_{j}$
is rigid.\end{cor}
\begin{proof}
We proceed by induction, observing that the result is trivial for
$n=1$. For the inductive step, the hypothesis implies $\Sigma_{m_{1}}(m)=-p_{1}$,
and therefore $c_{m}=c_{m-p_{1}m_{1}}$ by Corollary \ref{cor:m-pm}.
\end{proof}

\section{Mending Fractured Immersions}

We will analyze in more detail the main result of the preceding section.
This analysis is a necessary preliminary for the results in Section
\ref{sec:Reduction-of-the-intersection-problem}. Let us fix a tree
$T$ and an immersion $\varphi$ of $T$ which maps all the triple
vertices of $T$ to $\triangle_{r}$. Let $m'=m_{\varphi}$ be the
corresponding measure in $\mathcal{M}_{r}$, and let $\nu\in\mathcal{M}_{r}$
be another measure. Fix for the moment an edge $e_{0}$ in $T$, and
orient all the other edges of $T$ away from $e_{0}$. We define for
every vertex $V$ of $T$ a number $\delta_{e_{0}}(V,\nu)$. Assume
first that $V$ has order 2 and the corresponding edges are $AV,VB$,
oriented toward $B$. Setting $A'=\varphi(A),V'=\varphi(V),B'=\varphi(B)$,
we set\[
\delta_{e_{0}}(V,\nu)=\nu(V'X),\]
 where $X$ is on the right side of $A'V'$, and $\varangle XV'B'=60^{\circ}$.
On the other hand, if $V$ has order 3 and the corresponding edges
are $AV,BV,CV$, with $AV$ oriented toward $V$, then \[
\delta_{e_{0}}(V,\nu)=\nu(V'X),\]
 where $A'=\varphi(A)$, $V'=\varphi(V)$, and $X$ are collinear.
When $V$ is one of the endpoints of $e_{0}$, we orient $e_{0}$
toward that endpoint in this definition. Theorem \ref{thm:sigma(m,m)}
can now be given a more precise form.
\begin{thm}
\label{thm:arbitrary-starting-edge}With $\nu$ and $m'$ as above,
we have\[
\Sigma_{m'}(\nu)+\nu(\varphi(e_{0}))=\sum_{V}\delta_{e_{0}}(V,\nu).\]
\end{thm}
\begin{proof}
The easiest way to see this is to cut $e_{0}$ in half, and orient
the two halves away from its midpoint $Y$. Construct a fractured
immersion $\psi$ of $T$ as in the proof of Theorem \ref{thm:sigma(m,m)}.
For this immersion we have $\delta_{\psi}(V)=\delta_{e_{0}}(V,m)$
for each $V$, and $\delta_{g}(Y)=-\nu(\varphi(e_{0}))$.
\end{proof}
In the preceding proof, when $\varphi(e_{0})$ is not contained in
the support of $\nu$, the edge $\varphi(e_{0})$ is simply translated
along with the white puzzle piece which contains it. For our next
result, it will be important that $m'$ be a rigid measure and $\varphi(e_{0})$
be a root edge for the measure $m'$ with $m'(\varphi(e_{0}))=1$.
With this choice, Lemma \ref{lem:local-structure-of-skeletons} implies
the equality\[
\delta_{e_{0}}(V,m')=0\]
for every vertex $V$.

Let $T$ be a tree, and let $AV,VB$ be two edges meeting at a vertex
$V$ of order 2. One can \emph{stretch} the tree to a tree $T'$ replacing
$V$ by a path $V_{1}V_{2}\cdots V_{k}$ of consecutive vertices of
order 2 and the edges $AV$ and $BV$ are replaced by $AV_{1}$ and
$BV_{k}$. Analogously, if $AV,BV,CV$ are three edges meeting at
$V$, we can stretch $T$ by replacing $V$ with a `tripod' formed
by edges $V_{1}V_{2}\cdots V_{i}X$, $W_{1}W_{2}\cdots W_{j}X$, $U_{1}U_{2}\cdots U_{k}X$,
where all new vertices except $X$ have order 2, and $AV,BV,CV$ are
replaced by $AV_{1},BW_{1},CU_{1}.$ If $T'$ is obtained from $T$
by a finite number of such stretch operations, we will say that $T'$
is a \emph{stretch} of $T$. If $\varphi$ is an immersion of a stretch
$T'$ of $T$, the restriction of $f$ to the original edges of $T$
determines a fractured immersion $\psi$ of $T$ with the property
that $\delta_{\psi}(V)=0$ for every vertex $V$ of $T$. Such a fractured
immersion of $T$ will be said to be \emph{stretchable}. If $\psi$
is a stretchable fractured immersion and it is obtained as the restriction
of an immersion $\varphi$, we will also write $m_{\psi}$ for the
measure $m_{\varphi}$. The condition $\delta_{\psi}(V)=0$ for all
$V$ is not sufficient for stretchability. For instance, assume that
$V$ has degree $2$, $AV,VB$ are the two adjacent edges mapped by
$\psi$ to $A'V'$ and $V''B'$. The condition $\delta_{\psi}(V)=0$
implies that the points $A',B',V',V''$ are collinear, but stretchability
requires that $V'$ and $V''$ should be between $A'$ and $B'$;
the distance from $V'$ to $V''$ is precisely the number of additional
edges one must add at the point $V$. Similarly, if $AV,BV,CV$ are
mapped to $A'V',B'V'',C'V'''$, the condition $\delta_{\psi}(V)=0$
implies that these three lines intersect in a point $Z$, and stretchability
requires that $V'$ (resp. $V'',V'''$) be between $A'$ (resp. $B',C'$)
and $Z$.

Part of the following argument (namely, the case $q=0$) amounts to
a simplified proof of \cite[Theorem 4.3]{bcdlt}.
\begin{thm}
\label{thm:stretchability}Let $\mu,m'\in\mathcal{M}_{r}$, where
$m'$ is an extremal rigid measure assigning unit density to its root
edges; in particular $m'=m_{\varphi}$ for some immersion $\varphi$
of a tree $T$. Assume further that $\Sigma_{m'}(\mu)=0$. Denote
by $\alpha_{\ell},\beta_{\ell},\gamma_{\ell}$ and $\alpha_{\ell}',\beta'_{\ell},\gamma'_{\ell}$
the exit densities of $\mu$ and $m'$, respectively. There exists
a stretchable fractured immersion $\psi$ of $T$ such that 
\begin{enumerate}
\item all the limits of $\psi$ at discontinuity points are contained in
$\triangle_{r+\omega(\mu)}$, 
\item the exit densities $\widetilde{\alpha}_{i}'$ of the corresponding
measure $\widetilde{m}'=m_{\psi}\in\mathcal{M}_{r+\omega(\mu)}$ are
only different from zero for $i=\ell+\sum_{s=0}^{\ell-1}\alpha{}_{s}$,
$\ell=1,2,\dots,r$, in which case $\widetilde{\alpha}_{i}'=\alpha_{\ell}'$,
with similar formulas for $\widetilde{\beta}_{i}'$ and $\widetilde{\gamma}_{i}'$.
\end{enumerate}
\end{thm}
\begin{proof}
Denote by $q$ the largest integer with the property that $qm'\le\mu$,
and set $\nu=\mu-qm'$. It is clear that the support of $m'$ is not
contained in the support of $\nu$. 

Assume first that $q=0$, and choose an edge $e_{0}$ such that $\varphi(e_{0})$
is contained in $\triangle_{r}$ and $\nu(\varphi(e_{0}))=0$. Theorem
\ref{thm:arbitrary-starting-edge} implies that $\delta_{e_{0}}(V,\nu)=0$
for every vertex $V$ of $T$. Orient all the edges of $T$ away from
$e_{0}$, and construct a fractured immersion $\psi$ of $T$ by attaching
each $\varphi(e)$ to the white puzzle piece of $\nu$ on its right.
The condition $\delta_{e_{0}}(V)=0$ insures that $\psi$ is stretchable
at $V$, so that (1) holds. Since all the ends of $T$ are oriented
outward, condition (2) is satisfied as well.

Consider now the case $q>0$, fix an edge $e_{0}$ such that $\varphi(e_{0})$
is a root edge of $m'$ contained in $\triangle_{r}$, and orient
all the edges away from $e_{0}$. Give $e_{0}$ either orientation,
and construct a fractured immersion $\psi_{0}$ of $T$ by attaching
each $\varphi(e)$ to the white puzzle piece of $\mu=\nu+qm'$ on
its right. To conclude the proof, it will suffice to construct a stretchable
fractured immersion $\psi$ which coincides with $\psi_{0}$ on the
ends of $T$. Note that $\psi_{0}(e)$ is now an edge of a dark gray
parallelogram whose other side has length $\mu(e)$. We construct
$\psi(e)$ by moving $\psi_{0}(e)$ inside this parallelogram a number
of units equal to\[
\sum_{V\ge e}\delta_{e_{0}}(V,\nu)=\sum_{V\ge e}\delta_{e_{0}}(V,\mu),\]
away from the white piece to which $\psi_{0}(e)$ was attached, where
the sum is extended over the vertices $V$ which are descendants of
$e$ in the chosen orientation. In other words, the sum is extended
over those vertices $V$ for which the shortest path from $e_{0}$
to $V$ passes through $e$. It is important to note that $\psi(e)$
really is contained in this (closed) gray parallelogram, and for this
purpose it suffices to show that\[
\sum_{V\ge e}\delta_{e_{0}}(V,\mu)\le\mu(e).\]
This follows from the fact that $\delta_{e_{0}}(V,\mu)=\delta_{e}(V,\mu)$
if $V\ge e$, and therefore\[
\sum_{V\ge e}\delta_{e_{0}}(V,\mu)=\sum_{V\ge e}\delta_{e}(V,\mu)\le\sum_{V}\delta_{e}(V,\mu)=\mu(e)\]
by Theorem \ref{thm:arbitrary-starting-edge}, since $\Sigma_{m'}(\mu)=0$.
Also observe that the position of $\psi(e_{0})$ does not depend on
the orientation chosen for $e_{0}$ because (with either orientation)\[
\sum_{V\ge e_{0}}\delta_{e_{0}}(V,\mu)+\sum_{V\le e_{0}}\delta_{e_{0}}(V,\mu)=\sum_{V}\delta_{e_{0}}(V,\mu)=\mu(e_{0}),\]
and this is precisely the width of the dark gray parallelogram of
which $\psi_{0}(e_{0})$ is a side. It remains now to verify that
$\psi$ is stretchable. Consider first two edges $e_{1}=AB,e_{2}=BC$
adjacent to a vertex $B$ of order 2, oriented toward $B$ and $C$
respectively. Assume that $\varphi(e_{1})=A'B',$ $\varphi(e_{2})=B'C'$,
and the small edge $B'X$ is on the right of $A'B'$ such that $\varangle XB'C'=60^{\circ}$.
We have then\[
\sum_{V\ge e}\delta_{e_{0}}(V,\mu)=\mu(B'X)+\sum_{V\ge f}\delta_{e_{0}}(V,\mu),\]
so that $\psi_{0}(e_{1})$ must be moved left $\mu(B'X)$ more units
than $\psi_{0}(e_{1})$. This is precisely what is needed to align
$\psi(e)$ and $\psi(f)$, as illustrated in the figure below, where
the solid lines represent $\psi_{0}(e_{1})$ and $\psi_{0}(e_{2})$,
the dashed lines represent $\psi(e)$ and $\psi(f)$, and the dotted
line represents the range of the stretch of $\psi$.

\begin{center}\includegraphics{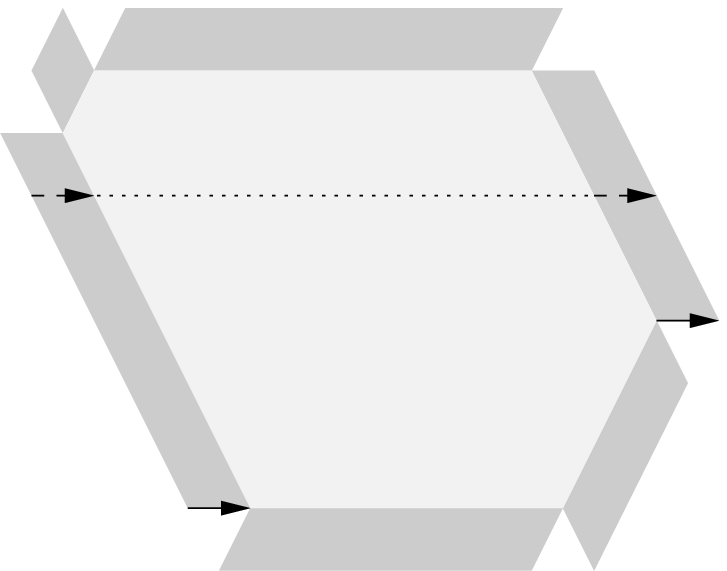}\end{center}

Assume now that $e_{1}=AV,e_{2}=BV,e_{3}=CV$ are three edges adjacent
to $V$, such that $e_{1}$ is oriented toward $V$. These edges are
mapped by $\varphi$ to $A'V',B'V',C'V'$. Let $V'X$ be the small
edge opposite $A'V'$.We have\[
\sum_{V\ge e_{1}}\delta_{e_{0}}(V,\mu)=\mu(V'X)+\sum_{V\ge e_{2}}\delta_{e_{0}}(V,\mu)+\sum_{V\ge e_{3}}\delta_{e_{0}}(V,\mu).\]

\begin{center}\includegraphics{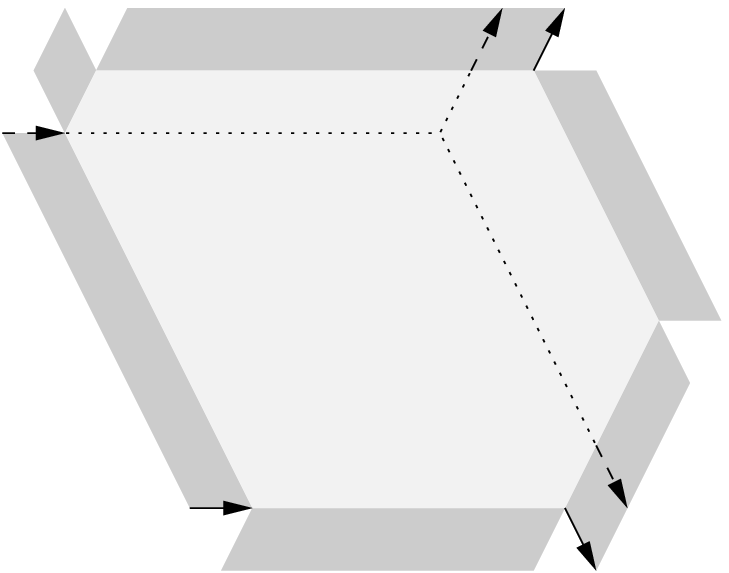}\end{center}

\noindent This relation is precisely what is needed to insure that
the break of $\psi$ at $V$ is stretchable, as in the illustration.
\end{proof}
\noindent The preceding theorem produces a measure $\widetilde{m}'$
which is again a rigid tree measure. Indeed, $\widetilde{m}'$ has
the same nonzero densities as $m'$, and therefore $\Sigma_{\widetilde{m}'}(\widetilde{m}')=\Sigma_{m'}(m')=-1$.
In fact, it is easy to see that $\widetilde{m}'$ is homologous to
$m'$ in the sense defined in \cite{bcdlt} and discussed in the following
section. Indeed, using the notation in the proof above, this follows
because two edges $e,e'$ of $T$ such that $\varphi(e)=\varphi(e')$
will satisfy\[
\sum_{V\ge e}\delta_{e_{0}}(V,m)=\sum_{V\ge e'}\delta_{e_{0}}(V,m),\]
and therefore their translates $\psi(e)$ and $\psi(e')$ will coincide
as well.

\section{Reduction of the Intersection Problem\label{sec:Reduction-of-the-intersection-problem}}

We are now ready to discuss the reduction procedures mentioned in
the introduction. We recall first some facts from \cite{bcdlt}. Fix
a measure $m\in\mathcal{M}_{r}$ with integer densities. A point $A_{\ell}$
(resp. $B_{\ell},C_{\ell}$) is called an \emph{attachment point}
of $m$ if $\ell\ge1$ and $m(A_{\ell}X_{\ell})>0$ (resp. $m(B_{\ell}Y_{\ell})>0,m(C_{\ell}Z_{\ell})>0$).
We denote by ${\rm att}_{I}(m)$ (resp. ${\rm att}_{J}(m)$, ${\rm att}_{K}(m)$)
the collection of indices $\ell\in\{1,2,\dots,r\}$ such that $A_{\ell}$
(resp. $B_{\ell},C_{\ell})$ is an attachment point for $m$.

Let now $I_{m},J_{m},K_{m}\subset\{1,2,\dots,n=r+\omega(m)\}$ be
the sets of cardinality $r$ defined by (\ref{eq:15}). The index
$i_{\ell}\in I_{m}$ (resp. $j_{\ell}\in J_{m},k_{\ell}\in K_{m}$)
is called an \emph{attachment index }for\emph{ $m$ }if $A_{\ell}$
(resp. $B_{\ell},C_{\ell}$) is an attachment point. We denote by
$I_{m}^{\text{att}},J_{m}^{\text{att}},K_{m}^{\text{att}}$ the collections
of attachment indices; thus $I_{m}^{\text{att}}=\{i_{\ell}:\ell\in{\rm att}_{I}(m)\}$.
Assume further that we are given flags $\mathcal{E},\mathcal{F},\mathcal{G}$
in $\mathbb{C}^{n}$. The spaces\emph{ }$\{\mathbb{E}_{i_{\ell}}:\ell\in I_{m}^{\text{att}}\}$,
$\{\mathbb{F}_{j_{\ell}}:\ell\in J_{m}^{\text{att}}\}$, $\{\mathbb{G}_{k_{\ell}}:\ell\in K_{m}^{\text{att}}\}$
are called the \emph{attachment spaces }of $m$. 

Let now $\widetilde{m}\in\mathcal{M}_{\widetilde{r}}$ be a second
measure with integer densities. The measures $m$ and $\widetilde{m}$
are said to be \emph{homologous} if there is a bijection between the
white piece edges determined by the support of $m$ and those determined
by the support of $\widetilde{m}$ such that corresponding edges are
parallel, and incident edges correspond to incident edges (the intersection
point being the one dictated by the correspondence of the edges).
If $m$ and $\widetilde{m}$ are homologous, there clearly exist order
preserving bijection $\varphi_{I}:I_{m}^{\text{att}}\to I_{\widetilde{m}}^{\text{att}}$,
$\varphi_{J}:J_{m}^{\text{att}}\to J_{\widetilde{m}}^{\text{att}}$,
$\varphi_{K}:K_{m}^{\text{att}}\to K_{\widetilde{m}}^{\text{att}}$.

Also recall that a lattice polynomial of a collection $\mathcal{X}=(\mathbb{X}_{\nu})_{\nu\in N}$
of spaces is defined inductively by the requirements that
\begin{enumerate}
\item for each $\nu$, the expression $P_{\nu}(\mathcal{X})=\mathbb{X}_{\nu}$
is a lattice polynomial, and
\item if $P(\mathcal{X})$ and $Q(\mathcal{X})$ are lattice polynomials,
then $(P(\mathcal{X}))+(Q(\mathcal{X}))$ and $(P(\mathcal{X}))\cap(Q(\mathcal{X}))$
are also lattice polynomials.
\end{enumerate}
More formally, lattice polynomials should be defined as elements of
an abstract lattice generated by a set of variables indexed by $N$.
One can then substitute subspaces for the variables to obtain a new
subspace. This gives the proper meaning to the last statement in the
next theorem.

The following result is a reformulation of results in \cite{bcdlt}.
The fact that the lattice polynomial is essentially the same for all
homologous measures is not explicitly stated there, but it is easily
verified using the argument of \cite[Proposition 5.1]{bcdlt}.
\begin{thm}
\label{thm:lattice-poly}Assume that $m\in\mathcal{M}_{r}$ is a rigid
measure with integer densities, and $\mathcal{E},\mathcal{F},\mathcal{G}$
are flags in $\mathbb{C}^{n}$, $n=r+\omega(m)$. There exists a lattice
polynomial $P_{m}$ of the attachment spaces of $m$ such that generically\[
P_{m}(\mathbb{E}_{i},\mathbb{F}_{j},\mathbb{G}_{k}:i\in I_{m}^{{\rm att}},j\in J_{m}^{{\rm att}},k\in K_{m}^{{\rm att}})\in\mathfrak{S}(\mathcal{E},I_{m})\cap\mathfrak{S}(\mathcal{F},J_{m})\cap\mathfrak{S}(\mathcal{G},K_{m}).\]
Moreover, if $\widetilde{m}\in\mathcal{M}_{\widetilde{r}}$ is homologous
to $m$ and $\widetilde{\mathcal{E}},\widetilde{\mathcal{F}},\widetilde{\mathcal{G}}$
are flags in $\mathbb{C}^{\widetilde{n}}$,\[
P_{\widetilde{m}}(\widetilde{\mathbb{E}}_{\varphi_{I}(i)},\widetilde{\mathbb{F}}_{\varphi_{J}(j)},\widetilde{\mathbb{G}}_{\varphi_{K}(k)}:i\in I_{m}^{{\rm att}},j\in J_{m}^{{\rm att}},k\in K_{m}^{{\rm att}})\]
equals\[
P_{m}(\widetilde{\mathbb{E}}_{\varphi_{I}(i)},\widetilde{\mathbb{F}}_{\varphi_{J}(j)},\widetilde{\mathbb{G}}_{\varphi_{K}(k)}:i\in I_{m}^{{\rm att}},j\in J_{m}^{{\rm att}},k\in K_{m}^{{\rm att}}).\]

\end{thm}
Given a measure (rigid or not) $m\in\mathcal{M}_{r}$ and an extremal,
rigid tree measure $m'\in\mathcal{M}_{r}$, we will be able to apply
a reduction of the Schubert problem associated to $m$ provided that
$\Sigma_{m'}(m)=-p<0$. More precisely, the Schubert problem will
be reduced to the corresponding problem for the measure $m-pm'$ (which
satisfies $c_{m-pm'}=c_{m}$ by Corollary \ref{cor:m-pm}) in a space
$\mathbb{X}$ of dimension $n-p\omega(m')$. The space $\mathbb{X}$
is obtained by applying the lattice polynomial $P_{m'}$ to the attachment
spaces of $m$ corresponding to the attachment points of $m'$. The
following result describes the procedure in detail. The argument is
essentially contained in \cite[Proposition 5.1]{bcdlt}, but we include
it here for completeness, and as a practical recipe. Observe that
$\Sigma_{m'}(m)+\omega'n$ is precisely the sum (\ref{eq:reduction-witness})
mentioned in our initial discussion of reductions.
\begin{thm}
\label{thm:reduction-procedure}Let $m,m'\in\mathcal{M}_{r}$ be two
measures with integer densities such that $m'$ is a rigid tree measure,
and $\Sigma_{m'}(m)=-p<0$. Denote by $i_{\ell},j_{\ell},k_{\ell}$,
$\ell=1,2,\dots,r$, the elements of $I=I_{m},J=J_{m},K=K_{m}$, respectively.
Given generic flags $\mathcal{E},\mathcal{F},\mathcal{G}$ in $\mathbb{C}^{n}$,
$n=r+\omega(m)$, the space\[
\mathbb{X}=P_{m'}(\mathbb{E}_{i_{x}},\mathbb{F}_{j_{y}},\mathbb{G}_{k_{z}}:x\in{\rm att}_{I}(m'),y\in{\rm att}_{J}(m'),z\in{\rm att}_{K}(m'))\]
has dimension $n-p\omega(m')$. Moreover, denote by $\mathcal{E}'$
the flag in $\mathbb{X}$ obtained by intersecting the spaces in $\mathcal{E}$
with $\mathbb{X}$ and discarding repeating spaces, with similar definitions
for $\mathcal{F}'$ and $\mathcal{G}'$. Then we have \[
\mathfrak{S}(\mathcal{E}',I_{m-pm'})\cap\mathfrak{S}(\mathcal{F}',J_{m-pm'})\cap\mathfrak{S}(\mathcal{G}',K_{m-pm'})\mathfrak{\subset S}(\mathcal{E},I_{m})\cap\mathfrak{S}(\mathcal{F},J_{m})\cap\mathfrak{S}(\mathcal{G},K_{m}).\]
\end{thm}
\begin{proof}
Denote the exit densities of $m$ by $a_{\ell},b_{\ell},c_{\ell}$,
$\ell=0,1,\dots,r$. Thus the elements of $I_{m},J_{m},K_{m}$ are
given by\[
i_{\ell}=\ell+\sum_{\ell'<\ell}a_{\ell'},\quad j_{\ell}=\ell+\sum_{\ell'<\ell}b_{\ell'},\quad k_{\ell}=\ell+\sum_{\ell'<\ell}c_{\ell'}\]
 for $\ell=1,2,\dots,r$. By Corollary \ref{cor:m-pm}, we can write
$m=pm'+\mu$, where $\mu\in\mathcal{M}_{r}$, and $\Sigma_{m'}(\mu)=0$.
Denote the exit densities of $\mu$ and $m'$ by $\alpha_{\ell},\beta_{\ell},\gamma_{\ell}$
and $\alpha'_{\ell},\beta'_{\ell},\gamma'_{\ell}$, respectively.
We have\[
a_{\ell}=\alpha_{\ell}+p\alpha'_{\ell},\quad b_{\ell}=\beta_{\ell}+p\beta'_{\ell},\quad c_{\ell}=\gamma_{\ell}+p\gamma'_{\ell}\]
for $\ell=0,1,\dots,r$. Theorem \ref{thm:stretchability} yields
a rigid tree measure $\widetilde{m}'\in\mathcal{M}_{r+\omega(\mu)}$,
homologous to $m'$, whose only possible nonzero exit densities are
$\widetilde{\alpha}_{i}'=\alpha'_{\ell}$ for\[
i=\ell+\sum_{k<\ell}\alpha_{k}=\ell+\sum_{k<\ell}(a_{k}-p\alpha'_{k}),\quad\ell=0,1,2,\dots,r,\]
with analogous formulas for $\widetilde{\beta}'_{i}$ and $\widetilde{\gamma}'_{i}$.
The set $I_{p\widetilde{m}'}$ of cardinality $r+\omega(\mu)$ has
elements\[
\widetilde{i}_{x}=x+\sum_{y<x}p\widetilde{\alpha}'_{y},\quad x=1,2,\dots,r+\omega(\mu).\]
In particular\begin{equation}
\widetilde{i}_{x}=i_{\ell}\text{ when }x=\ell+\sum_{k<\ell}\alpha_{k},\label{eq:ixiell}\end{equation}
with similar formulas for $\widetilde{j}_{x}$ and $\widetilde{k}_{x}$.
We deduce that the attachment spaces of $p\widetilde{m}'$ are precisely
\[
\{\mathbb{E}_{i_{x}},\mathbb{F}_{j_{y}},\mathbb{G}_{k_{z}}:x\in{\rm att}_{I}(m'),y\in{\rm att}_{J}(m')z\in{\rm att}_{K}(m')\}.\]
Now, the measure $p\widetilde{m}'$ is homologous to $m'$, and therefore
Theorem \ref{thm:lattice-poly} implies that the space $\mathbb{X}$
in our statement belongs generically to the intersection\[
\mathfrak{S}(\mathcal{E},I_{p\widetilde{m}'})\cap\mathfrak{S}(\mathcal{F},J_{p\widetilde{m}'})\cap\mathfrak{S}(\mathcal{G},K_{p\widetilde{m}'}).\]
Relation (\ref{eq:ixiell}) implies that\[
\dim(\mathbb{X}\cap\mathbb{E}_{i_{\ell}})\ge\ell+\sum_{k<\ell}\alpha_{k},\quad\ell=1,2,\dots,r,\]
with similar estimates for $\dim(\mathbb{X}\cap\mathbb{F}_{j_{\ell}})$
and $\dim(\mathbb{X}\cap\mathbb{G}_{k_{\ell}})$. Thus, by intersecting
the spaces in the flags $\mathcal{E},\mathcal{F},\mathcal{G}$ with
$\mathbb{X}$ we obtain (after eliminating repeating spaces) flags
$\mathcal{E}',\mathcal{F}',\mathcal{G}'$ in $\mathbb{X}$ with the
property that\begin{equation}
\mathbb{E}_{x}^{\prime}\subset\mathbb{X}\cap\mathbb{E}_{i_{\ell}}\text{ for }x=\ell+\sum_{k<\ell}\alpha_{k},\label{eq:flag-incl}\end{equation}
and similarly for $\mathcal{F}'$ and $\mathcal{G}'$.

Note now that $\omega(m-pm')=n-p\omega(m')=\dim(\mathbb{X})$, and
therefore it makes sense to solve the Schubert problem associated
with this measure and the flags $\mathcal{E}',\mathcal{F}',\mathcal{G}'$.
To conclude the proof, let $\mathbb{M}$ be a space in the intersection\[
\mathfrak{S}(\mathcal{E}',I_{m-pm'})\cap\mathfrak{S}(\mathcal{F}',J_{m-pm'})\cap\mathfrak{S}(\mathcal{G}',K_{m-pm'}).\]
To see that $\mathbb{M}$ belongs to\[
\mathfrak{S}(\mathcal{E},I_{m})\cap\mathfrak{S}(\mathcal{F},J_{m})\cap\mathfrak{S}(\mathcal{G},K_{m}),\]
observe that the $\ell$th element of $I_{m-pm'}$ is equal to $i=\ell+\sum_{k<\ell}\alpha_{k}$,
so that\begin{eqnarray*}
\dim(\mathbb{M}\cap\mathbb{E}_{i_{\ell}}) & = & \dim(\mathbb{M}\cap(\mathbb{X}\cap\mathbb{E}_{i_{\ell}}))\\
 & \ge & \dim(\mathbb{M}\cap\mathbb{E}_{i}^{\prime})\ge\ell,\end{eqnarray*}
where we used (\ref{eq:flag-incl}) in the first inequality.
\end{proof}
When the measure $m$ is itself rigid, it was shown in \cite{bcdlt}
that it is possible to choose $m'$ so that $\Sigma_{m'}(m)=-p<0$,
and $m-pm'$ has strictly smaller support than $m$. Therefore repeated
applications of these reduction procedures  eventually yield an explicit
solution of the intersection problem.

As an illustration, we will see how to deduce the two kinds of reductions
mentioned in the introduction. First, consider a measure $m'$ with
$\omega(m')=1$. There are only three nonzero exit densities $\alpha'_{x}=\beta'_{y}=\gamma'_{z}=1$,
and we must have $x+y+z=r$; see the first triangle in the figure
below. The sum \[
\Sigma_{m'}(\mu)=i_{x}+j_{y}+k_{z}-n\]
corresponds to the original reductions in \cite{th-th}, and a reduction
can be applied when this sum is negative. The relevant lattice polynomial
is\[
P_{m'}(\mathbb{E},\mathbb{F},\mathbb{G})=\mathbb{E}+\mathbb{F}+\mathbb{G},\]
thus yielding the formula mentioned in the introduction. The second
reduction outlined in the introduction corresponds with a measure
$m'$ satisfying $\omega(m')=2$ whose support is shown in the second
triangle below.

\begin{center}
\includegraphics[scale=0.5]{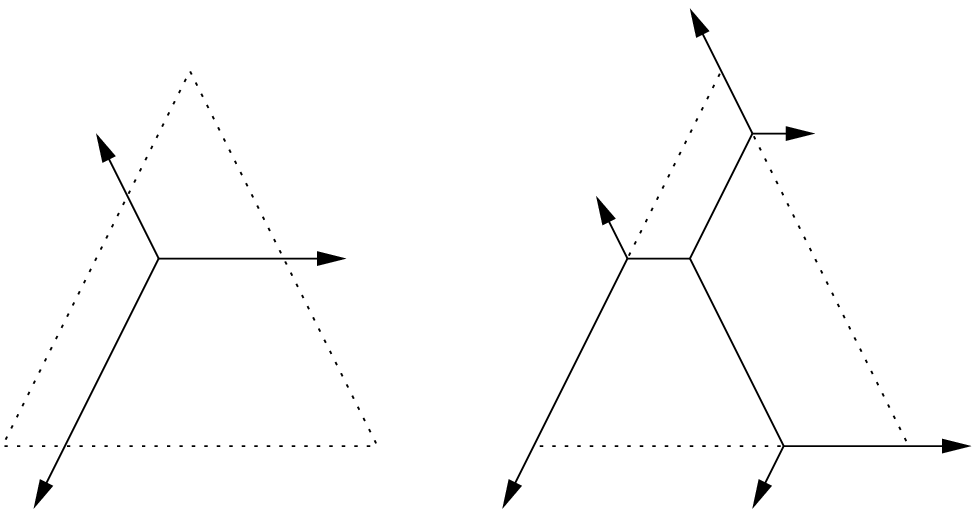}
\par\end{center}

\noindent There are now six exit densities equal to 1, but three of
them are $\alpha_{0}',\beta'_{0},\gamma'_{0}$, which do not correspond
to attachment points. The others are $\alpha'_{x},\beta'_{y},\gamma'_{z}$,
where the numbers $x,y,z$ are the lengths of the dotted segments
in the boundary of $\triangle_{r}$. Clearly $x+y+z=2r$, and\[
\Sigma_{m'}(\mu)=i_{x}+j_{y}+k_{z}-2n.\]
This time the lattice polynomial is\[
P_{m'}(\mathbb{E},\mathbb{F},\mathbb{G})=(\mathbb{E}\cap\mathbb{F})+(\mathbb{F}\cap\mathbb{G})+(\mathbb{G}\cap\mathbb{E}).\]

We conclude this section with an analysis of the measures in $\mathcal{M}_{3}$
which do not allow any of the reductions outlined above. There are
11 extremal measures in $\mathcal{M}_{3}$, and all of them are rigid
tree measures. Their supports are depicted below.

\begin{center}
\includegraphics[scale=0.6]{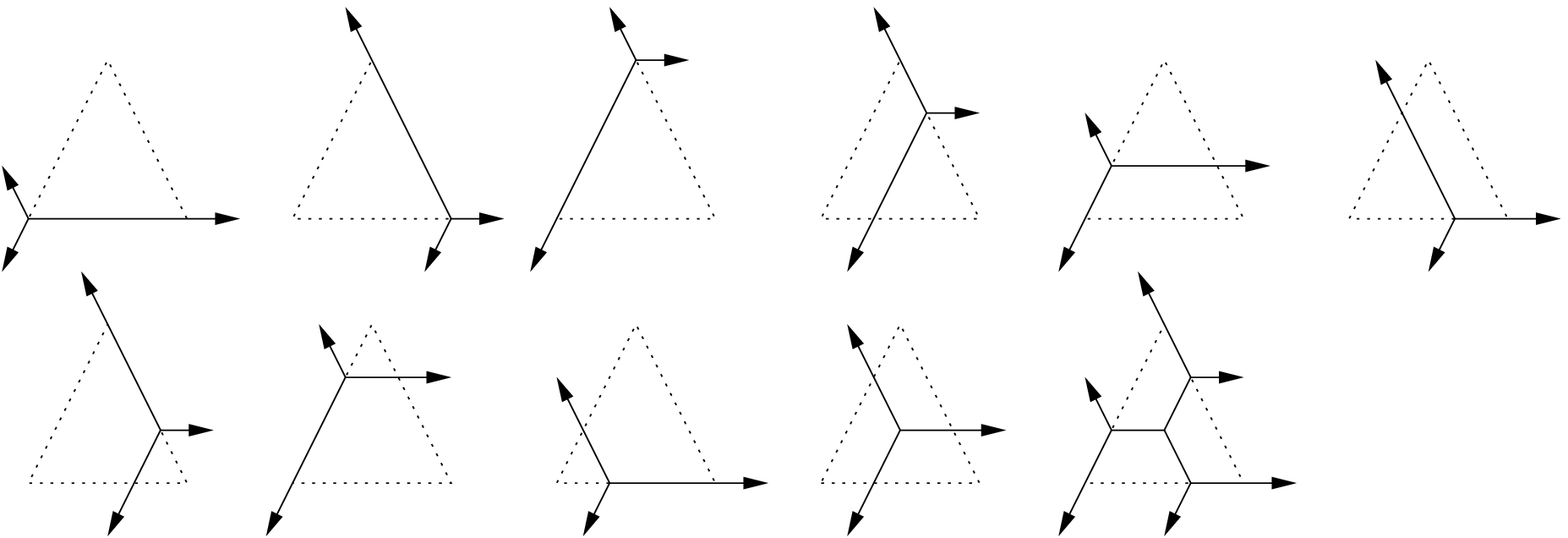}
\par\end{center}

\noindent Let us call these tree measures $\mu_{1},\mu_{2},\mu_{3},\nu_{1},\nu_{2},\nu_{3},\rho_{1},\rho_{2},\rho_{3},\tau_{1},$
and $\tau_{2}$. In addition to the equalities $\Sigma_{\mu}(\mu)=-1$,
the only other nonzero values for $\Sigma_{\mu}(\nu)$ with $\mu,\nu$
among these measures are equal to one. These are: $\Sigma_{\nu_{j}}(\mu_{j}),$
$\Sigma_{\rho_{j}}(\mu_{j})$, $\Sigma_{\tau_{1}}(\mu_{j})$ for $j=1,2,3$,
and the three cycles $\Sigma_{\nu_{j}}(\nu_{j+1})$, $\Sigma_{\rho_{j}}(\rho_{j-1})$,
and $\Sigma_{\tau_{j}}(\tau_{j+1})$. An arbitrary measure $m\in\mathcal{M}_{3}$
with integer densities can be written as\[
m=\sum_{j=1}^{3}(a_{j}\mu_{j}+b_{j}\nu_{j}+c_{j}\rho_{j})+d_{1}\tau_{1}+d_{2}\tau_{2},\]
where the coefficients $a_{j},b_{j},c_{j},d_{j}$ are nonnegative
integers. Note that $\Sigma_{\mu_{j}}(m)=-a_{j}$, $\Sigma_{\nu_{j}}(m)=a_{j}+b_{j+1}-b_{j}$,
$\Sigma_{\rho_{j}}(m)=a_{j}+c_{j-1}-c_{j}$, $\Sigma_{\tau_{2}}(m)=d_{1}-d_{2}$,
and $\Sigma_{\tau_{1}}(m)=a_{1}+a_{2}+a_{3}+d_{2}-d_{1}$. A reduction
is possible unless $a_{1}=a_{2}=a_{3}=0$, $b_{1}=b_{2}=b_{3}$, $c_{1}=c_{2}=c_{3}$,
and $d_{1}=d_{2}$. Observe also that $\tau_{1}+\tau_{2}=\nu_{1}+\nu_{2}+\nu_{3}$,
and this measure has the same exit densities as $\rho_{1}+\rho_{2}+\rho_{3}$.
Thus the only intersection problems which cannot be reduced with our
methods arise from measures of the form $d(\tau_{1}+\tau_{2})$ for
some integer $d>0$. The measure $m$ is rigid if and only if $b_{1}b_{2}b_{3}=c_{1}c_{2}c_{3}=d_{1}d_{2}=0.$
The first ten of the rigid tree measures above correspond with the
reductions considered in \cite{th-th} and \cite{CoDy-reduction}.
It should be noted that this analysis can be applied, via the duality
described in \cite{bcdlt}, to the analysis of measures $m$ with
$\omega(m)=3$. The intersection problems for such measures can be
reduced to duals of measures of the form $d(\tau_{1}+\tau_{2})$.

A similar analysis can be carried out for $r=4$ and $r=5$, but with
many more tree measures. Indeed, for $r\le5$ all extremal measures
in $\mathcal{M}_{r}$ are rigid. For $r=6$ there are already some
extremal measures which are not tree measures, though their exit densities
coincide with those of a sum of extremal rigid measures. An example
is provided below, where all solid edges have unit density. The two
resulting measures have the same exit densities, but only the first
one is extremal; the second one is the sum of three extremal measures.

\begin{center}
\includegraphics[scale=0.5]{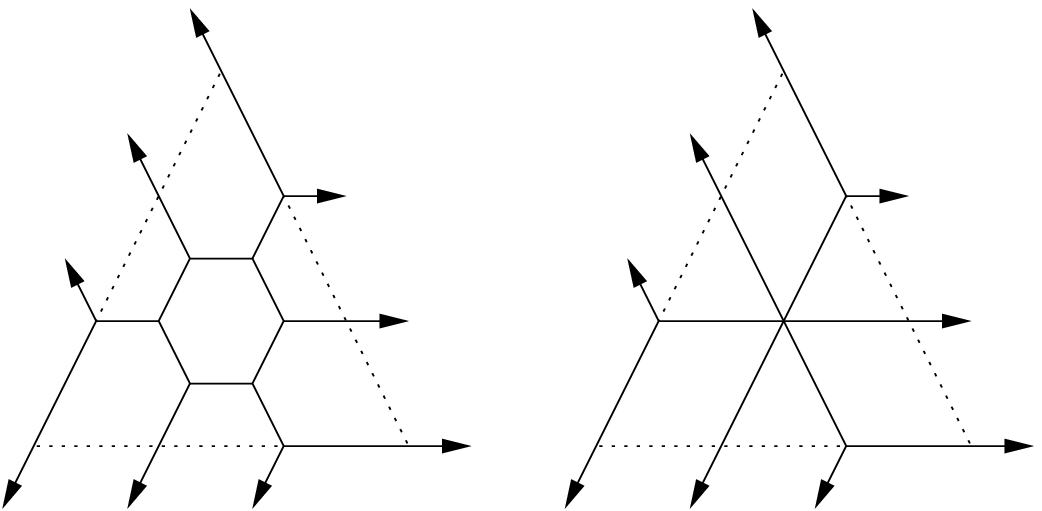}
\par\end{center}

For larger values of $r$, there exist tree extremal measures which
are not rigid, and  do not have the same exit densities as any sum
of extremal rigid measures. The support of such a tree measure $m$
is pictured below.

\begin{center}
\includegraphics[scale=0.5]{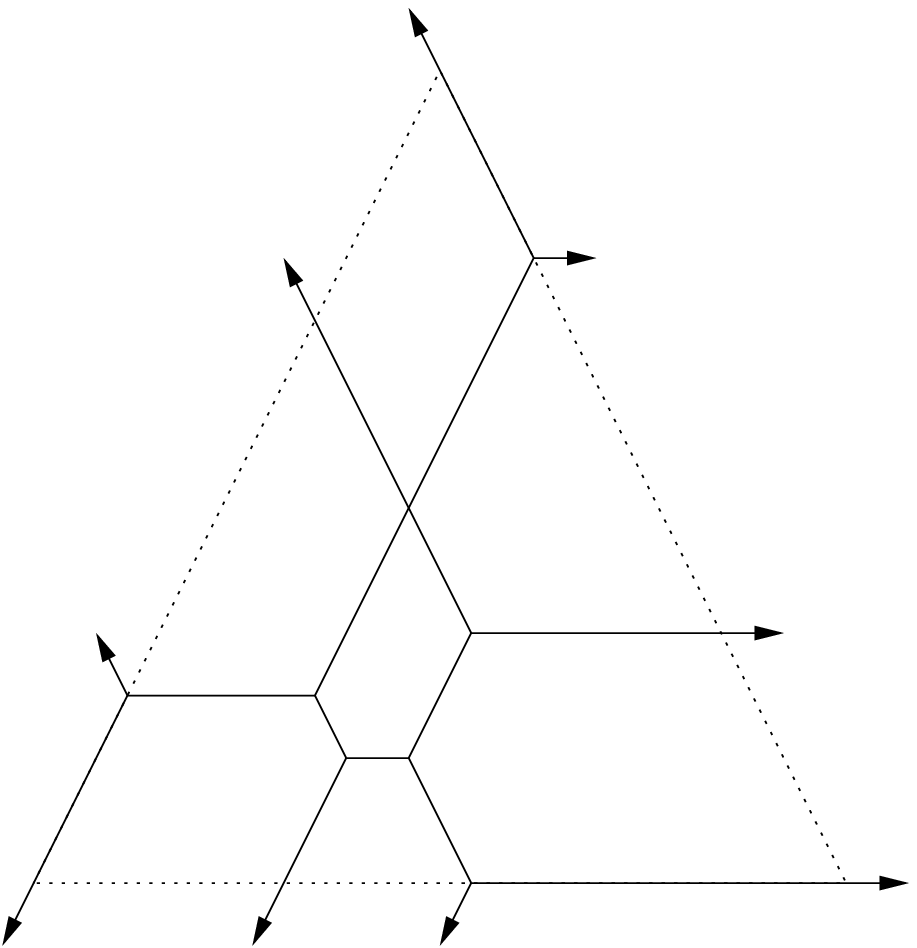}
\par\end{center}

\noindent Here $r=13$, and the exit points are $A_{0},A_{4},A_{10},B_{0},B_{4},B_{7},C_{0},C_{4},$
and $C_{10}$. It is easy to verify that one cannot find among these
points $A_{x},B_{y},C_{z}$  such that $x+y+z=13$, and therefore
the exit densities do not majorize the exit densities of any measure
$\mu$ with $\omega(\mu)=1$. Since $\omega(m)=3$, it follows that
the exit densities of $m$ do not majorize those of any rigid tree
measure.

\section{An Arboretum of Rigid Tree Measures}

The reduction procedure described in the preceding section requires
knowledge of the rigid tree measures in $\mathcal{M}_{r}$, and one
might hope that a complete description of these is available. We are
not aware of the existence of such a description, but we will use
Theorem \ref{thm:sigma(m,m)} to study those rigid measures which
have three nonzero exit densities on each side of $\triangle_{r}$.
Assume thus that the rigid tree measure $m\in\mathcal{M}_{r}$ has
weigt $\omega$, and nonzero densities $\alpha,\alpha',\alpha''$
in the NW direction, $\beta,\beta',\beta''$ in the SW direction,
and $\gamma,\gamma',\gamma''$ in the E direction. These integers
must satisfy\begin{equation}
\alpha+\alpha'+\alpha''=\beta+\beta'+\beta''=\gamma+\gamma'+\gamma''=\omega,\label{eq:weight-sec6}\end{equation}
and\begin{equation}
\alpha^{2}+\alpha^{\prime2}+\alpha^{\prime\prime2}+\beta^{2}+\beta^{\prime2}+\beta^{\prime\prime2}+\gamma^{2}+\gamma^{\prime2}+\gamma^{\prime\prime2}=\omega^{2}+2\label{eq:sum-of-squares}\end{equation}
by Theorem \ref{thm:sigma(m,m)}(3). Assume first that $\omega=3k+1$
for some integer $k\ge1$. The smallest value allowed by (\ref{eq:weight-sec6})
for the sum \[
\alpha^{2}+\alpha^{\prime2}+\alpha^{\prime\prime2}+\beta^{2}+\beta^{\prime2}+\beta^{\prime\prime2}+\gamma^{2}+\gamma^{\prime2}+\gamma^{\prime\prime2}\]
s achieved when the weights on each side are $k,k$ and $k+1$, and
that value is precisely $\omega^{2}+2$. Thus (\ref{eq:sum-of-squares})
implies that the weights on each side have precisely these values
(in some order). Similarly, when $\omega=3k+2$, the densities on
each side must be $k,k+1$ and $k+1$. When $\omega=3(k+1),$ relation
(\ref{eq:sum-of-squares}) implies\[
\alpha^{2}+\alpha^{\prime2}+\alpha^{\prime\prime2}+\beta^{2}+\beta^{\prime2}+\beta^{\prime\prime2}+\gamma^{2}+\gamma^{\prime2}+\gamma^{\prime\prime2}\ge9(k+1)^{2}=\omega^{2},\]
 with equality achieved only when all the exit densities are equal
to $k+1$. It follows easily from (\ref{eq:sum-of-squares}) that
on two sides the exit densities will all be equal to $k+1$, while
on the remaining side they must be $k,k+1,k+2$. We will now produce
actual examples of rigid tree measures with three nonzero exit densities
in each direction, and with all possible values of $\omega$. A first
series of examples is described in the following figure.

\begin{center}
\includegraphics[scale=0.8]{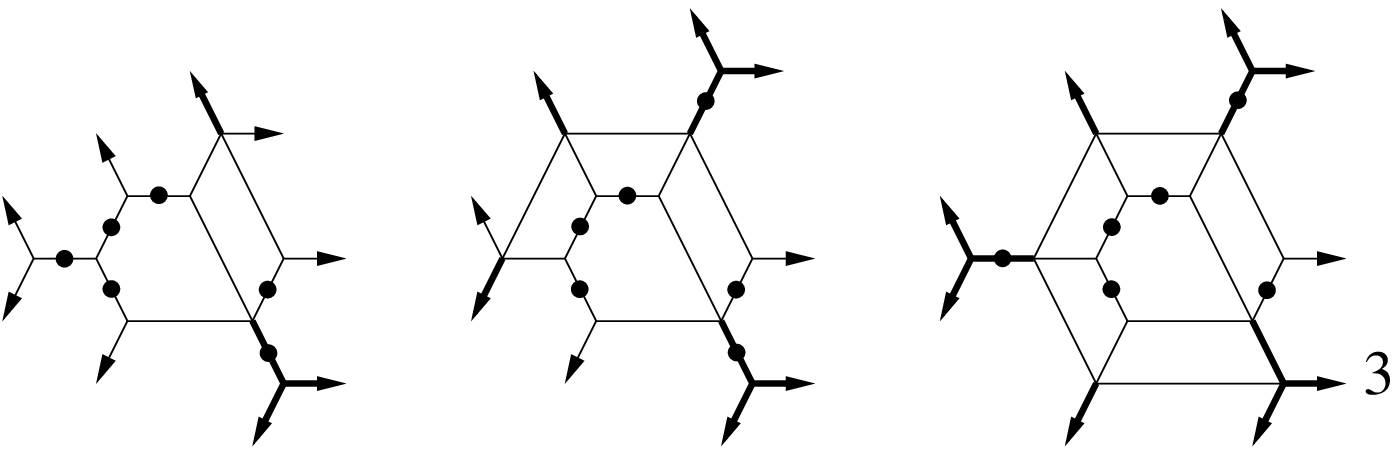}
\par\end{center}

\noindent The thinner edges have density one, and the thicker ones
have density two, except for one exit density which is equal to three,
as labeled. Other such measures can be obtained by applying $120^{\circ}$
rotations to these measures, or symmetries about a horizontal line.
Another way to obtain new meassures is to change the lengths of the
edges indicated by a dot. These lengths can be chosen arbitrarily;
here is an example of this procedure applied to the second measure
above.

\begin{center}
\includegraphics[scale=0.6]{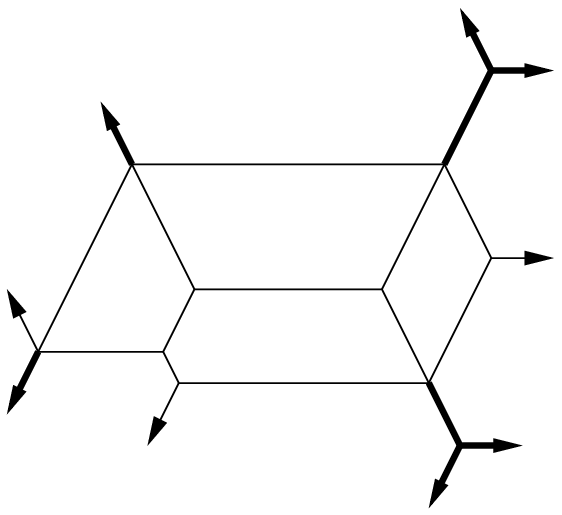}
\par\end{center}

The three measures above provide examples with $\omega=3k+1,3k+2$
and $3(k+1)$ when $k=1$. For larger values of $k$ one must continue
the spiral pattern. A second series of examples is illustrated below.

\begin{center}
\includegraphics[scale=0.8]{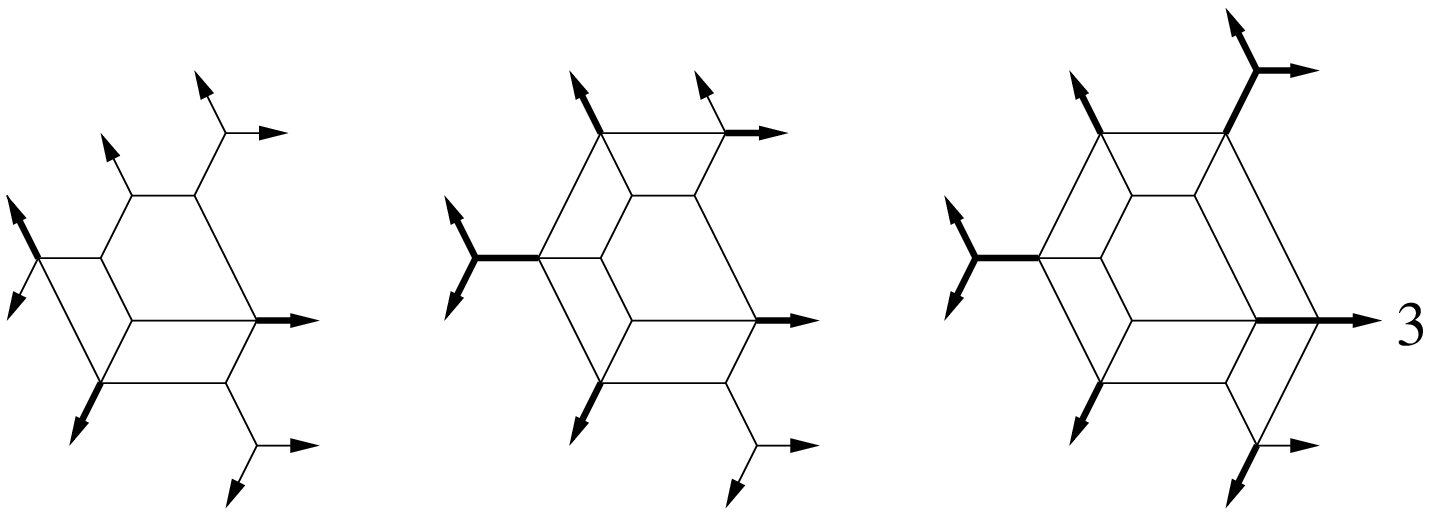}
\par\end{center}

\noindent As in the first series of examples, these measures can be
rotated by multiples of $120^{\circ}$, and reflected in a horizontal
line. Their shapes can also be changed by modifying arbitrarily the
lengths of six of the edges. Again, the spiral can be continued to
yield examples with weights $3k+1,3k+2$ and $3(k+1)$ for all intergers
$k\ge1$.

A third series of examples is illustrated next.

\begin{center}
\includegraphics[scale=0.6]{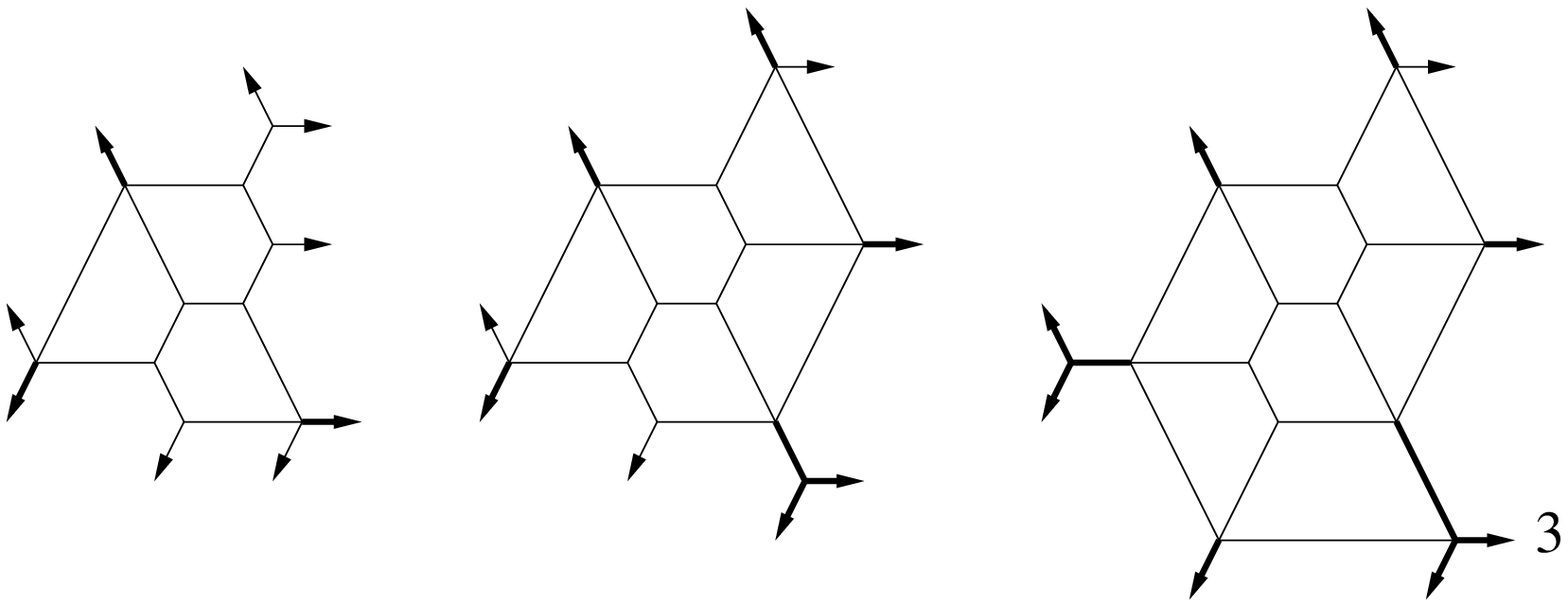}
\par\end{center}

\noindent Note that this series has two spiral arms. To obtain measures
with higher weight one proceeds by alternately increasing each spiral
by $1/3$ of a complete turn.

When $\omega(m)=3k+1$ there is one more series of measures which
have greater symmetry. The first two in the series are pictured below.

\begin{center}
\includegraphics[scale=0.6]{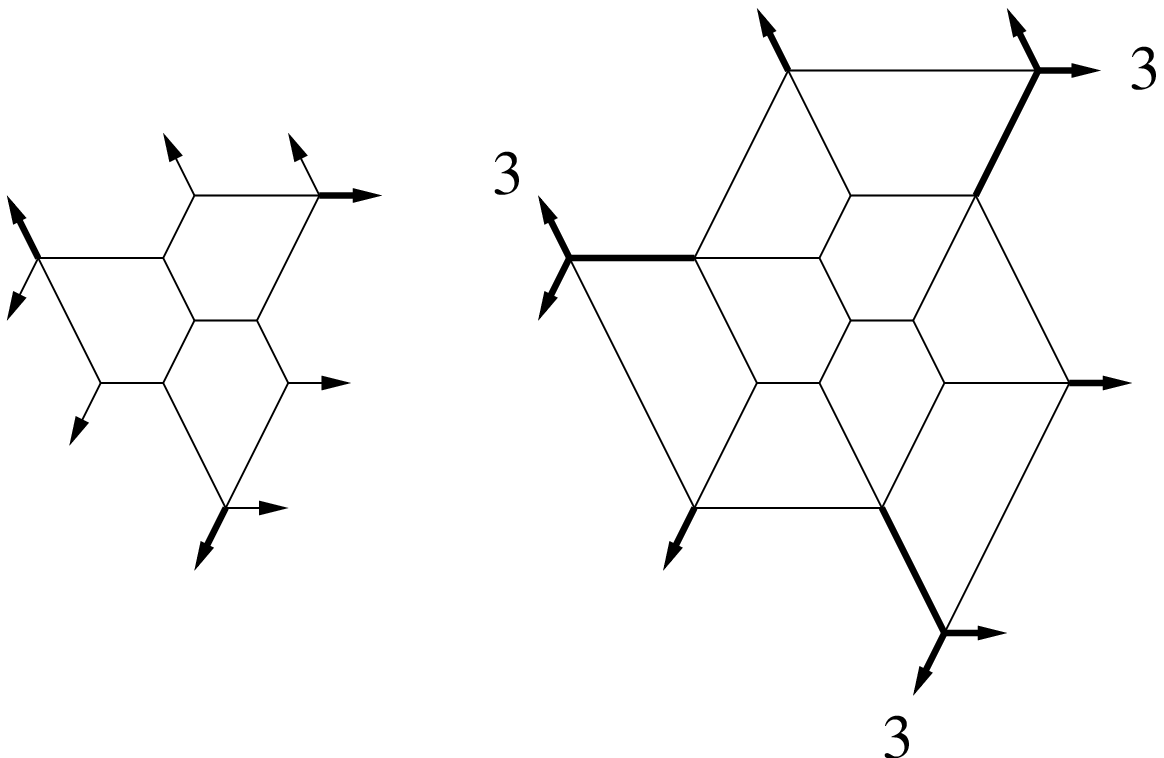}
\par\end{center}

\noindent These measures are invariant under $120^{\circ}$ rotations,
but not under reflection relative to a horizontal line.

A similar series is available for $\omega=3k+2$.

\begin{center}
\includegraphics[scale=0.6]{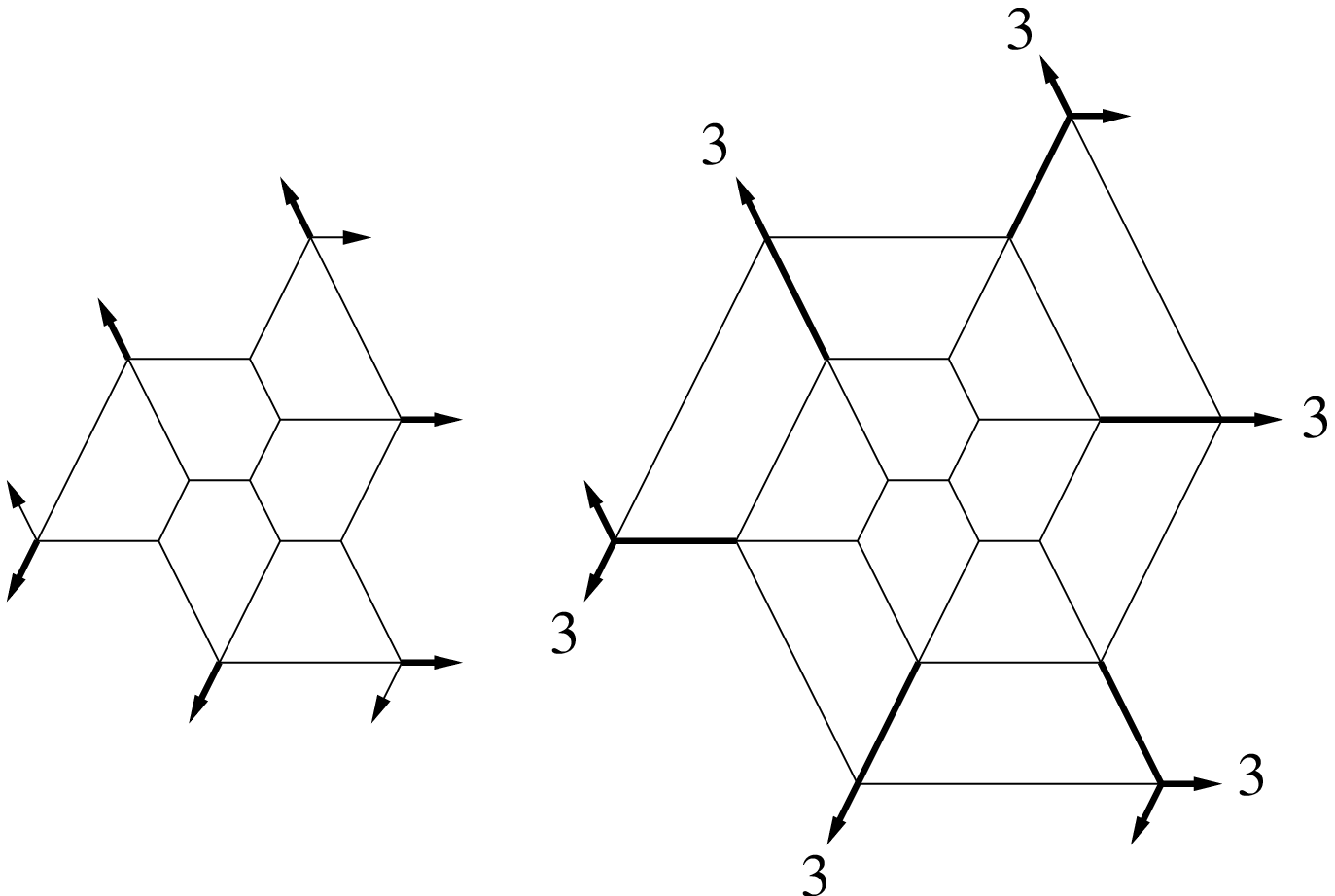}
\par\end{center}

Some of these examples have versions for $k=0$, though in that case
there will be fewer than three nonzero exit densities in some direction.
Using duality of measures, it can be shown that the measures described
above (along with their rotations, reflections and stretched versions)
are the only measures with exactly three nonzero exit densities in
each direction. Thus, for instance, there are no rigid tree measures
whose exit densities are (in counterclockwise order, starting with
$\alpha$) $(k+1,k+1,k),$ $(k+1,k,k+1)$ and $(k,k+1,k+1)$ or $(k+1,k,k)$,
$(k,k,k+1)$ and $(k,k+1,k)$.

\end{document}

%% file: uvw.pstex_t
\begin{picture}(0,0)%
\includegraphics{uvw}%
\end{picture}%
\setlength{\unitlength}{3947sp}%
\begingroup\makeatletter\ifx\SetFigFont\undefined%
\gdef\SetFigFont#1#2#3#4#5{%
  \reset@font\fontsize{#1}{#2pt}%
  \fontfamily{#3}\fontseries{#4}\fontshape{#5}%
  \selectfont}%
\fi\endgroup%
\begin{picture}(777,791)(1036,-470)
\put(1051, 14){\makebox(0,0)[lb]{\smash{{\SetFigFont{12}{14.4}{\familydefault}{\mddefault}{\updefault}{$u$}%
}}}}
\put(1756, 14){\makebox(0,0)[lb]{\smash{{\SetFigFont{12}{14.4}{\familydefault}{\mddefault}{\updefault}{$w$}%
}}}}
\put(1336,-406){\makebox(0,0)[lb]{\smash{{\SetFigFont{12}{14.4}{\familydefault}{\mddefault}{\updefault}{$v$}%
}}}}
\end{picture}%

%% file: BCD-cyclic.pstex_t
\begin{picture}(0,0)%
\includegraphics{BCD-cyclic}%
\end{picture}%
\setlength{\unitlength}{3947sp}%
\begingroup\makeatletter\ifx\SetFigFont\undefined%
\gdef\SetFigFont#1#2#3#4#5{%
  \reset@font\fontsize{#1}{#2pt}%
  \fontfamily{#3}\fontseries{#4}\fontshape{#5}%
  \selectfont}%
\fi\endgroup%
\begin{picture}(1380,1261)(886,-1775)
\put(1801,-661){\makebox(0,0)[lb]{\smash{{\SetFigFont{10}{10}{\familydefault}{\mddefault}{\updefault}{$B'$}%
}}}}
\put(2251,-1200){\makebox(0,0)[lb]{\smash{{\SetFigFont{10}{10}{\familydefault}{\mddefault}{\updefault}{$C$}%
}}}}
\put(1201,-1711){\makebox(0,0)[lb]{\smash{{\SetFigFont{10}{10}{\familydefault}{\mddefault}{\updefault}{$B$}%
}}}}
\put(901,-1200){\makebox(0,0)[lb]{\smash{{\SetFigFont{10}{10}{\familydefault}{\mddefault}{\updefault}{$C'$}%
}}}}
\put(1201,-661){\makebox(0,0)[lb]{\smash{{\SetFigFont{10}{10}{\familydefault}{\mddefault}{\updefault}{$D$}%
}}}}
\put(1951,-1711){\makebox(0,0)[lb]{\smash{{\SetFigFont{10}{10}{\familydefault}{\mddefault}{\updefault}{$D'$}%
}}}}
\put(1701,-1063){\makebox(0,0)[lb]{\smash{{\SetFigFont{10}{10}{\familydefault}{\mddefault}{\updefault}{$A$}%
}}}}
\end{picture}%

%% file: basic-triangle.pstex_t
\begin{picture}(0,0)%
\includegraphics{basic-triangle}%
\end{picture}%
\setlength{\unitlength}{3947sp}%
\begingroup\makeatletter\ifx\SetFigFont\undefined%
\gdef\SetFigFont#1#2#3#4#5{%
  \reset@font\fontsize{#1}{#2pt}%
  \fontfamily{#3}\fontseries{#4}\fontshape{#5}%
  \selectfont}%
\fi\endgroup%
\begin{picture}(2303,2328)(2536,-2305)
\put(4251,-1336){\makebox(0,0)[lb]{\smash{{\SetFigFont{10}{10}{\familydefault}{\mddefault}{\updefault}{$Z_2$}%
}}}}
\put(4401,-1636){\makebox(0,0)[lb]{\smash{{\SetFigFont{10}{10}{\familydefault}{\mddefault}{\updefault}{$Z_1$}%
}}}}
\put(4551,-1936){\makebox(0,0)[lb]{\smash{{\SetFigFont{10}{10}{\familydefault}{\mddefault}{\updefault}{$Z_0$}%
}}}}
\put(4251,-1986){\makebox(0,0)[lb]{\smash{{\SetFigFont{10}{10}{\familydefault}{\mddefault}{\updefault}{$C_0$}%
}}}}
\put(3051,-1986){\makebox(0,0)[lb]{\smash{{\SetFigFont{10}{10}{\familydefault}{\mddefault}{\updefault}{$B_1$}%
}}}}
\put(2751,-1986){\makebox(0,0)[lb]{\smash{{\SetFigFont{10}{10}{\familydefault}{\mddefault}{\updefault}{$B_0$}%
}}}}
\put(3201,-2236){\makebox(0,0)[lb]{\smash{{\SetFigFont{10}{10}{\familydefault}{\mddefault}{\updefault}{$Y_2$}%
}}}}
\put(2901,-2236){\makebox(0,0)[lb]{\smash{{\SetFigFont{10}{10}{\familydefault}{\mddefault}{\updefault}{$Y_1$}%
}}}}
\put(2601,-2236){\makebox(0,0)[lb]{\smash{{\SetFigFont{10}{10}{\familydefault}{\mddefault}{\updefault}{$Y_0$}%
}}}}
\put(3501,-436){\makebox(0,0)[lb]{\smash{{\SetFigFont{10}{10}{\familydefault}{\mddefault}{\updefault}{$A_0$}%
}}}}
\put(3351,-136){\makebox(0,0)[lb]{\smash{{\SetFigFont{10}{10}{\familydefault}{\mddefault}{\updefault}{$X_0$}%
}}}}
\put(3201,-436){\makebox(0,0)[lb]{\smash{{\SetFigFont{10}{10}{\familydefault}{\mddefault}{\updefault}{$X_1$}%
}}}}
\put(4101,-1636){\makebox(0,0)[lb]{\smash{{\SetFigFont{10}{10}{\familydefault}{\mddefault}{\updefault}{$C_1$}%
}}}}
\put(3351,-736){\makebox(0,0)[lb]{\smash{{\SetFigFont{10}{10}{\familydefault}{\mddefault}{\updefault}{$A_1$}%
}}}}
\put(3051,-736){\makebox(0,0)[lb]{\smash{{\SetFigFont{10}{10}{\familydefault}{\mddefault}{\updefault}{$X_2$}%
}}}}
\end{picture}%